\documentclass[12pt]{amsart}

\usepackage{amsmath}
\usepackage{amsfonts}
\usepackage{amssymb}
\usepackage{hyperref}

\DeclareMathOperator{\sgn}{sgn}

\DeclareMathOperator{\tr}{tr}
\DeclareMathOperator{\Real}{Re}

\DeclareMathOperator{\R}{\mathbb{R}}
\DeclareMathOperator{\T}{\mathbb{T}}
\DeclareMathOperator{\Z}{\mathbb{Z}}

\newtheorem{theorem}{Theorem}[section]
\newtheorem{lemma}[theorem]{Lemma}
\newtheorem{corollary}[theorem]{Corollary}

\theoremstyle{definition}
\newtheorem{definition}[theorem]{Definition}

\theoremstyle{remark}
\newtheorem{remark}[theorem]{Remark}

\numberwithin{equation}{section}

\usepackage{hyperref}

\begin{document}

%The date used records that last time
%that the file was edited.
%That requires changing the next line appropriately.
\def\now{24 December 2021}

\title[\now \hfill 
Factoring minimal majorants
\hfill]
{Minimal Fourier majorants in~$L^p$}
\author
[Fournier and Vrecko]
{John J. F. Fournier$^1$
%\footnote{Research 
%of the first author was 
%supported by NSERC grant 4822.}
and Dean Vrecko$^2$
}
\address{$^1$Department of Mathematics\\
University of British Columbia,
Vancouver, Canada V6T 1Z2}
\email{fournier@math.ubc.ca}
\address{$^2$Department of
Emergency Medicine\\
University of British Columbia,
Vancouver, Canada V5Z 1M9}
\email{deanvrecko@alumni.ubc.ca}
\thanks{Results partially announced at the June 2001 meeting
of the Canadian Mathematical Society.}
\thanks{The research of the first author
was supported by NSERC grant 4822.}
\thanks{The research of the second author
was supported by an NSERC USRSA award.}
%\thanks{Research partially supported by NSERC grant 4822.}
%, 47L20}

\date{\now}

\begin{abstract}
Denote the
coefficients in the complex form of the  Fourier series of a
function~$f$ on the interval~$[-\pi, \pi)$ by~$\hat f(n)$.
%Use the symbol~$\hat f(n)$ for the~$n$-th coefficient in the complex form of the  Fourier series of a function~$f$ on the interval~$[-\pi, \pi)$. 
It is known that if~$p = 2j/(2j-1)$ for some
integer~$j>0$, then for each function~$f$ in~$L^p$ there exists 
another
%a
function $F$ in~$L^p$ that majorizes~$f$ in the sense that~$\hat
F(n) \ge |\hat f(n)|$ for all~$n$,
and for which~$\|F\|_p \le \|f\|_p$.
%but that also satisfies~$\|F\|_p \le \|f\|_p$.
When~$j > 1$, the
%The
existence proofs for such small majorants do not provide constructions of them, but
%it is known that
%Moreover, 
there is a unique
%such
majorant of minimal~$L^p$ norm.
We 
modify
%combine 
previous existence proofs
% and a modern proof
to say more about
%to get
%a
%some 
%new 
%restriction
%restrictions 
%on
the form of
that
%the minimal
majorant.
%Our method extends to many other settings.
% the context of
%Fourier transforms
%on~$\R^m$ and other locally compact abelian groups.
%It also applies to 
%convolution on discrete 
%groups,
%nonabelian 
%abelian or not,
%groups, 
%and to
%Schatten norms on matrices.
%and to noncommutative~$L^p$ norms associated with convolution
%on discrete nonabelian groups.
%In the cases where the method works, it gives more information with no more effort than previous approaches to these questions.
%about~$L^p$ norms.
\end{abstract}

\maketitle

\markleft{\hfill Fournier and Vrecko\hfill \now}
%That controls the header on the even-numbered pages. The headers on odd-numbered pages are controlled by the running title specified above.

\section{Introduction}
\label{sec:intro}

%{\allowdisplaybreaks

Thus~$\hat f(n) = (1/2\pi)\int_{-\pi}^\pi f(\theta)e^{-in\theta}\,d\theta$.
%Thus
%\[\hat f(n) =
%\frac{1}{2\pi}\int_{-\pi}^\pi f(\theta)e^{-in\theta}\,d\theta.
%\]
Call~$F$ a {\em majorant}
of~$f$,
%for~$f$,
and~$f$ a {\em minorant}
of~$F$,
%for~$F$,
when~$|\hat f(n)| \le \hat F(n)$
for all integers~$n$. In that case,~$\|f\|_2 \le \|F\|_2$; also, if~$j$ is
an integer greater than~$1$, then $F^j$ majorizes~$f^j$, and hence
\[(\|f\|_{2j})^{2j} = (\|f^j\|_2)^2 \le (\|F^j\|_2)^2 =
(\|F\|_{2j})^{2j}.\]
Finally,~$\|f\|_\infty \le \|F\|_\infty$ when~$F$ majorizes~$f$.

This pattern does not persist for other 
exponents.
%indices.
Hardy and Littlewood~\cite{HaL}
considered the {\em upper majorant property},
asserting that there is a
%, gave an example of trigonometric
%polynomials~$f$ and~$F$ with the property that~$|\hat f| \le \hat
%F$ but~$\|f\|_3 > \|F\|_3$. It was later shown \cite{Bo, Ba, Fo}
%that for each index~$p$ in the interval~$(0, \infty)$ except for
%the even integers, there is {\em no}
constant~$U(p)$
for which
%so that
\begin{equation}
\label{eq:upper}
\|f\|_p \le U(p) \|F\|_p
\end{equation}
whenever~$F$ majorizes~$f$.
They gave an example
%of trigonometric
%polynomials~$f$ and~$F$
%for which~$|\hat f| \le \hat
%F$ but~$\|f\|_3 > \|F\|_3$. So
showing that if
this property
holds in~$L^3$ then the constant~$U(3)$
must be strictly larger than~$1$.
%with constant~$1$.
Later work~\cite{Bo, Ba} revealed
%Other people~\cite{Bo, Ba} went on to show
%, 
%that,
%showed, 
%for each exponent~$p$ in the interval~$(0, \infty)$ except for the even integers
%,
that the property fails
completely
%, for all constants~$U(p)$,
for each exponent~$p$
%for 
%the exponents~$p$
%each exponent~$p$ 
in the
%in the
interval~$(0, \infty)$
that
is
%are
not
%except for the
even.
%even integers.
%That is, these~$L^p$ norms
%do {\em not} have the {\em upper majorant property}
%introduced in \cite{HaL}.
%even integers;
%those methods extend to the cases where~$0 < p < 1$.
%See~\cite[pp.~131--134]{Mt}, \cite{GR}, \cite{Gr},\cite{MS},~\cite{CKY} and~\cite{Ebe}
%for refinements
%of these results, 
%complements to them,
%and 
%for 
%connections with other 
%questions.
%problems.

Here we consider the {\em lower majorant property},
also introduced in~\cite{HaL}.
It holds when there is a constant~$L(p)$ so that
%\begin{description}
each function $f$ in $L^p$
has a majorant $F$
for which
%with
%\end{description}
\begin{equation}
\label{eq:lowmajor}
\|F\|_p \le L(p)\|f\|_p.
\end{equation}
This is clearly true when~$p=2$ with~$L(2) = 1$,
letting~$F$ be
%with~$F$ equal to
the {\em exact majorant}
of~$f$,
%of~$f$
that is~$\hat F = |\hat f|$.
%given by~$\hat F = |\hat f|$.
It also holds when~$p = 1$,
since one can factor a given function~$f$ 
as~$f_1f_2$,
%as~$f_1f_2$
where~$(\|f_1\|_2)\|f_2\|_2 = \|f\|_1$, 
%with~$(\|f_1\|_2)\|f_2\|_2 = \|f\|_1$, %replace~$f_1$ and~$f_2$
%by their
form the exact majorants~$F_1$ and~$F_2$ of~$f_1$ and~$f_2$,
%and ,
and then take their product~$F = F_1F_2$, which majorizes~$f$, 
and for which
%and satisfies
\[
\|F\|_1 \le (\|F_1\|_2)\|F_2\|_2 
= (\|f_1\|_2)\|f_2\|_2 = \|f\|_1.
\]
%More subtle factorizations extend this method to the
%Hardy
%space~$H^1$
%on the unit disk, and to~$H^1$ spaces in other contexts \cite{CRW}.
%Finally, (de)constructive methods
%Appropriate factorizations~\cite{Miyachi}
%extend this to~$H^p$ spaces where~$0 < p \le 1$.
%Constructions
%And constructions 
%using atoms~\cite[p.~584]{CW},
%\cite{Al}, \cite[p.~70]{BaeS} and \cite[p.~794]{FrJ}
%using atoms also
%apply to the same~$H^p$ spaces
%when~$0 < p \le 1$ 
%and to the corresponding Besov spaces.

When~$p \in (1, \infty)$,
%When~$p \in [1, \infty)$,
a simple duality argument \cite[Section 3]{Bo}
shows that if~$L^p$ has the lower majorant property,
then its dual space~$L^{p'}$ has the
upper majorant
property, and~$U(p') \le L(p)$.
% property with~$U(p') \le L(p)$.
By the
work cited 
%earlier,
above, 
this can only happen
%when~$p' = \infty$
when~$p'$ is an even integer,
%or~$p'$ is an even integer,
thus ruling out all exponents~$p$ in the
interval~$(1, \infty)$ except for~$p = 2$ and the
{\em 
special
%peculiar 
exponents}
for which~$p=2j/(2j-1)$,
%with~$p=2j/(2j-1)$,
where~$j$ is an integer strictly greater
than~$1$. 
When~$1 < p < \infty$, a
%A 
less simple
%more intricate 
duality argument~\cite{HaL} shows
that the upper majorant property for~$L^{p'}$ implies the lower
majorant property
%for~$L^p$;
for~$L^p$,
and that~$L(p) \le U(p')$.
%moreover~$L(p) \le U(p')$.
%and it shows that~$L(p) \le U(p')$.
%with~$L(p) \le U(p')$.
%Therefore,
%In particular,
So
the latter
%that
property
%it
holds for the special exponents,
%it holds,
and then~$L(p)=1$.
%with~$L(p)=1$.
%with~$L(p)=1$,
%for the special exponents.
%This was modified in~\cite{Ba}
%to yield that~$L(p) \le U(p')$ So~$L^p$ has
%the lower majorant
%this 
%property, with constant~$1$, for the 
%special
%peculiar 
%exponents.
%~$p = 4/3, 6/5, 8/7, \cdots$.

%That
The
duality proof
in~\cite{HaL}
and alternatives~\cite{Ob, DGLPQ, LuP} to it
do not 
specify
%include
%a construction of
a suitable majorant of a given function for these values of~$p$.
For various good reasons, those arguments covered
general
exponents~$p$, 
%exponents~$p$ and~$p'$, 
but
%even though 
the duality
%is now known to
mainly
has
%have
impact for the 
special
%peculiar 
values of~$p$.

We exploit this hindsight by
%concentrating on
analysing
those special cases and getting new information about
majorants with minimal~$L^{2j/(2j-1)}$ norm.
%This analysis 
%leads to a more elementary proof of the duality theorem. It 
%reduces the construction of majorants in~$L^{2j/(2j-1))}$ to norm-minimization problems in certain  simplices in~$L^{2j}$.
%compact subsets of finite-dimensional subspaces of~$L^{2j}$.
%We hope that this leads to constructions
%of majorants with control on their norms.
%Perhaps it will lead to  constructions of
%majorants of general~$L^p$ functions.
%The method suggests, however, that a better goal is the construction
%of functions~$G$ in~$L^{p'}$ with the properties
%specified below.
%It also leads to a
%more elementary proof of the duality theorem. That proof
%relies more on differentiation than on the use of duality.
For
%a
%the 
%peculiar exponent~$p$, 
an exponent~$p$ in the interval~$(1, \infty)$,
the majorants of a given function 
in~$L^p$
form a
%nonempty,
closed, convex subset.
%in~$L^p$. 
By the uniform convexity of the~$L^p$ norm,
if
that
subset
%set
is nonempty, then it
has a unique element of least norm.
%We add to previous analyses of that majorant as follows.

%\section{Main results}
%Our main results run as follows.
\begin{theorem}
\label{th:trig-th}
Let~$p = 2j/(2j-1)$ for some integer~$j>1$,
and 
%let~$f \in L^p$.
let~$f$ be a function in~$L^p$.
Then~$f$ has 
a majorant
%majorants 
with no larger~$L^p$ 
%norm,
norm.
%in~$L^p$, 
Its
%and its 
minimal
majorant~$F$
%majorant
in~$L^p$
factors as~$(\overline{G})^{j-1}G^j$, 
%has the form~$F= (\overline{G})^{j-1}G^j$, 
%form~$F=G^j(\overline{G})^{j-1}$, where~$G$ is a function in~$L^{p'}$
where~$G \in L^{2j}$
%where~$G \in L^{p'}$
and~$\hat G \ge 0$.
%with~$\hat G \ge 0$.
Moreover~$\hat G$ vanishes on the set 
%where $\widehat{G^j(\overline{G})^{j-1}}$
where~$\hat F > |\hat f|$,
and~$\hat G$
%and it
vanishes off
%on
the support of~$\hat f$.
%Moreover,~$\|F\|_p \le \|f\|_p$, and~$\hat G$ vanishes off the support of~$\hat f$.
\end{theorem}

\begin{corollary}
\label{th:supports}
If~$S$ is the set of frequencies of~$f$,
then the set of frequencies of the minimal majorant of~$f$
in~$L^{2j/(2j-1)}$ is included in the algebraic sum of~$j$
copies of~$S$ and~$j-1$ copies of~$-S$.
If~$f$ is a trigonometric polynomial,
%then
so is
its majorant of minimal norm in~$L^{2j/(2j-1)}$.
%with~$p = 2j/(2j-1)$.
\end{corollary}

The corollary is new, as is the inclusion of the support of~$\hat G$ in that of~$\hat f$.
Related
%Similar
conclusions appear in~\cite{FoArxiv}, which
presents
%is based on
work
that began later but ended earlier
Parts
of the
%The
following notion
arose in
earlier
%the early
%arises in various
duality proofs.
\begin{definition}
\label{def:conjugate}
Let~$p = 2j/(2j -1)$ for some integer~$j > 1$.
Given a function~$f$ in~$L^p$, call
a function~$H$
in~$L^{2j}$
%in~$L^{p'}$
%an~$L^{p'}$ function~$H$
a~$p$-{\em conjugate} of~$f$ if it has
the following properties:

\begin{enumerate}
\item
\label{en:positive} 
$\hat H \ge 0$.

\item
\label{en:majorize}
The
product~$J$
given by~$(\overline H)^{j-1}H^j$
%product~$J = (\overline H)^{j-1}H^j$
%product~$J = H^j(\overline H)^{j-1}$
%function~$J = H^j(\overline H)^{j-1}$
majorizes~$f$.

\item
\label{en:slackness}
$\hat H(n) = 0$ at all integers~$n$
%where~$\widehat{(\overline H)^{j-1}H^j(n)} > |\hat f(n)|$.
where~$\hat J(n) > |\hat f(n)|$.

\item
\label{en:support} 
$\hat H = 0$ off the support of~$\hat f$.

\end{enumerate}
\end{definition}

%Our proofs of Theorem~\ref{th:trig-th} will 
%also
%show that
%the norm-minimizing function~$G$ specified
%there
%in~Theorem~\ref{th:trig-th}
%is a~$p$-conjugate of~$f$.
%Except for part~\eqref{en:support},
%this is 
%essentially
%also 
%done
%shown 
%In
%in 
%the classical proof,
%except for part~\eqref{en:support}.
%this is also shown
%this is also implicit 
%in
%follows from 
%the classical duality proof. 
%and also from our method.
%One can also use the latter to show
%it is
%where it is also 
%shown
%which goes on to show
%that the 
%first
%other 
%three 
%conditions
%properties
%above 
%listed above
%The latter also shows 
%that if~$p'$ is even and~$H$ is a~$p$-conjugate of~$f$,
%where~$p'$ is even,
%then~$\|H\|_{p'} \le (\|f\|_p)^{p-1}$, and~$\|J\|_p \le \|f\|_p$.
%imply that~$\|J\|_p \le \|f\|_p$
%when~$p'$ is even;
%That method is used here in 
%we use
%the latter
%that 
%idea is used 
%in Section~\ref{sec:complements},
%and
%in a simpler way
%in formulas~\eqref{eq:Switch} and~\eqref{eq:minorant} below.
%In Section~\ref{sec:OtherProofs}, we explain how property~\eqref{en:support} follows from the others.
%other conditions.
%We also prove that property more directly in Section~\ref{sec:SpecialInd}.

It will be shown
%We will see
later that
property~\eqref{en:support} follows from the others, and that
the only~$p$-conjugate of such a function~$f$ is the function~$G$ arising in Theorem~\ref{th:trig-th}.
Property~\eqref{en:majorize}
%The third property above 
can be restated as the requirement that the
discrete convolution
%, on the integers, 
of~$j$ copies of the
nonnegative 
function~$\hat H$
%sequence~$\hat H$
with~$j-1$ copies of the reflected
transform~$\widehat{\overline H}$ mapping~$m$
to~$\overline{\hat H(-m)}$,
which is~$\hat H(-m)$ here,
%to~$\overline{\hat H(-m)}$
should
be no smaller than~$|\hat f|$. 
%Condition~\eqref{en:slackness} forces~$\hat H(n)$ to vanish at any indices~$n$ where that convolution is strictly greater than~$|\hat f(n)|$.

When~$f$ is a trigonometric polynomial, finding 
a~$p$-conjugate
%a~$p$-conjugate~$H$
%a minimal majorant~$H$
becomes a finite-dimensional nonlinear-programming problem. 
The object function is
the~$2j$-th power of the~$L^{2j}$ norm of~$H$;
this is also
%a homogeneous polynomial of degree~$2j$ in the coefficients of~$H$, namely
%The goal is to minimize
the square of the~$\ell^2$ norm of the convolution of~$j$ copies
of~$\hat H$.
The goal is to minimize
that homogeneous polynomial of degree~$2j$ in the coefficients of~$H$
%it
subject to the
constraints~(\ref{en:positive}) and (\ref{en:majorize}).
% and~\eqref{en:support}.
%with inequality constraints~(\ref{en:positive}) and~(\ref{en:majorize}) and slackness conditions~(\ref{en:slackness}).
The 
%duality argument in~\cite{HaL} shows that the
slackness condition~\eqref{en:slackness}
%these slackness condition 
is then necessary, as is condition~\eqref{en:support}.
%for the square of the~$\ell^2$ norm of the convolution of~$j$ copies of~$\hat H$ to have a minimum subject to the constraints~(\ref{en:positive}) and~(\ref{en:majorize}). This object function is a homogeneous polynomial of degree~$2j$ in the coefficients of~$H$.

The
%A modern
duality argument in~\cite{DGLPQ}
led us
%leads
to
%yields 
another
%a simpler
%second 
description
of~$H$ or~$G$
%of~$G$
%of the minimal~$p$-conjugate~$G$
%of the function~$G$ dual to the majorant of minimal~$L^{p'}$ norm
when~$f$ is a trigonometric polynomial.
Minimize the~$2j$-th power of the~$L^{2j}$ norm of a function~$g$
%Up to rescaling,~$G$ has minimal~$L^{p'}$ norm
subject to having
nonnegative coefficients that vanish off the support
of~$\hat f$, and
for which
%that satisfy the 
%linear
%\emph{linear} 
%constraint
\begin{equation}
\label{eq:mod-eq}
\sum_n
%\sum_{n \in \supp\left(\hat f\right)}
|\hat f(n)|\hat g(n) = 1.
\end{equation}
Then rescale~$g$ suitably to get~$G$.
%This dual 
%formulation 
%of the minimal majorant
%problem
%can be solved numerically for many
%trigonometric 
%polynomials~$f$.
%Our proofs
%are mostly based
%build
%on this idea.

%The lower majorant property for the space~$L^{2j/(2j-1)}$ then seems less surprising, since it is posed
%since it is posed in a  compact subset
%coset 
%of a finite-dimensional subspace of the space~$L^{2j}$, where questions about norms are easier.

%The dual formulation
%was the basis for the numerical work by Vrecko. It also
%leads to 
We elaborate on this in
%In
the next section,
where
we 
%introduce and 
study the
related
%relation
notion of 
partial
majorant
for all exponents in the interval~$(1, \infty)$.
%which we study in the next section.
%It is then easy
%We use our results about partial majorants
%to prove 
Theorem~\ref{th:trig-th} 
then follows
easily
in Section~\ref{sec:SpecialInd},
via a
certain
%dual norm
estimate in the special cases where~$p'$ is even.
%estimate.
%As noted earlier, the
%The
%main 
%novelty is the part about supports.
%In Section~\ref{sec:ConvoFacts}, we recast most of 
%that theorem
%Theorem~\ref{th:trig-th}
%as an elementary statement
%about convolution of finitely-supported sequences.
%In Section~\ref{sec:OtherProofs}, we 
We
explain
in Section~\ref{sec:OtherProofs}
%show 
how
the
%part about
restriction on 
the support of~$\hat G$ 
%part
%main novelty,
%about supports of coefficients
%it
%that part 
also 
follows
%from earlier results about minimal
%majorants,
%majorants:
%and
%we also explain
by earlier
methods.
%methods,
%and how
%how
%the properties
%of minimal partial majorants are easy to prove
%by the classical method.

%
%That
The
classical
method
%also
shows that if~$H$ is a~$p$-conjugate of~$f$,
then the~$L^p$ norm of~$
%and if~$J =
H^j(\overline H)^{j-1}$ is at most~$
%H^j(\overline H)^{j-1}$,
%then~$\|J\|_p \le \|f\|_p$.
\|f\|_p$.
In
%in 
Section~\ref{sec:complements},
%we
%We 
%use 
we extend this
%slightly, and
slightly to
prove the following
%statements, again
statements
when~$p = 2j/(2j -1)$ for some
integer~$j$ strictly greater than one.
%integer~$j > 1$.
%statements.

\begin
{theorem}
%{lemma}
\label{th:Dual Norm Inequality}
%Let~$p = 2j/(2j -1)$ for some integer~$j > 1$.
Given a function~$G$
in~$L^{2j}$
%in~$L^{p'}$
%Let~$G$ be a function in~$L^{p'}$
with nonnegative coefficients,
%and
let~$F = (\overline G)^{j-1}G^j$.
Let~$f$ be any function in~$L^{p}$
with the property that~$|\hat f| \ge \hat F$ on the support of~$\hat G$.
%Then
%\[
%\|F\|_{p} \le \|f\|_{p}.
%\]
Then~$\|f\|_{p} \ge \|F\|_{p}$.
%Then~$\|F\|_{p} \le \|f\|_{p}$.
\end
{theorem}
%{lemma}

\begin{corollary}
\label{th:minimalconjugate}
Let~$H$
%Let~$p=2j/(2j-1), and let~$H$ 
be a~$p$-conjugate of a function~$f$ in~$L^p$.
%in~$L^p$, where~$p' = 2j$.  
Then
%Then~$(\overline H)^{j-1}H^j$ is 
the minimal majorant of~$f$ in~$L^p$
is~$(\overline H)^{j-1}H^j$, and~$H$ is the only~$p$-conjugate of~$f$.
\end{corollary}

\begin{corollary}
\label{th:ownmajorant}
If~$H \in L^{2j}$
%If~$H \in L^{p'}$
and~$\hat H \ge 0$,
then~$(\overline H)^{j-1}H^j$ is its own minimal majorant in~$L^{p}$.
\end{corollary}
%the dual
%a slight extension of
%extend 
%the classical
%norm estimate
%is used
%and then deduce
%and use it
%to show
%slightly. 
%It follows 
%that each function in~$L^p$ has a unique~$p$-conjugate, and 
%that
%the
%functions that are minimal majorants of something
%all 
%minimal majorants 
%in~$L^p$ 
%are
%those
%the functions 
%that
%factor as products~$H^j\overline{H}^{j-1}$ where~$H \in L^{p'}$ and ~$\hat H \ge 0$.
%It follows that
%Moreover, 
%every such product is the minimal majorant in~$L^p$ of itself.
%We then compare
Comparing 
the supports of the coefficients of~$H$ and~$\overline{H}^{j-1}H^j$
%then
%to see
shows
that many functions in~$L^p$ with nonnegative coefficients are not minimal majorants in~$L^p$ of anything.
Similar reasoning leads to many
%we also 
%describe many other 
cases where~$G$ is not a multiple of
the exact
majorant
%majorant,~$E_f$ say,
of~$f$ 
because~$\hat G$ vanishes on part of the support of~$\hat f$.
Earlier numerical work had 
also
led us to
%other
such examples.

The
%functions
functions,~$f$ say,
that belong to~$L^{2j/(2j-1)}$
are those that factor as~$(\overline k)^{j-1}k^j$
where~$k \in L^{2j}$.
It turns out that the
%The
quantities~$\|F\|_p$ and~$\|f\|_p$
in Theorem~\ref{th:Dual Norm Inequality}
are equal
if and only if the exact majorant~$E_k$  of~$k$ also belongs
to~$L^{2j}$,
%to~$L^{2j}$
with
%and has
the same norm 
in that space
%in~$L^{2j}$
as~$k$.
%if and only if~$f$ factors as~$|g|^{2(j-1)}g$, 
%where~$g \in L^{2j}$,
%where~$g \in L^{p'}$,
%its exact majorant is~$G$,
%and~$\|G\|_{2j} = \|g\|_{2j}$.
%and~$\|G\|_{p'} = \|g\|_{p'}$.
%These are also the cases where~$F$ is the minimal majorant of~$f$, and~$\|F\|_p = \|f\|_p$.
Moreover,~$F =(\overline{E_k})^{j-1}(E_k)^j$
in those cases.
%In those cases,~$F =(\overline{E_k})^{j-1}(E_k)^j$.
A
standard
%suitable
thinness condition applied to
supports of
coefficients
%of functions
%in~$L^{p'}$
%Lacunarity also 
yields
many such examples.

\begin{remark}
%It is easy to see that if~$G \in L^{2j}$ and~$\hat G \ge 0$, then the product~$G^j(\overline{G})^{j-1}$ belongs to~$L^{2j/(2j-1)}$, and this product has nonnegative coefficients. It is implicit in previous proofs of the lower majorant property
%for~$L^{2j/(2j-1)}$
%was not previously noticed
%that these products suffice to majorize
%every function in~$L^{2j/(2j-1)}$.
%
%We will give a new proof that they suffice.
%and even to give the majorant of minimal~$L^{2j/(2j-1)}$ norm.
% and even to do so with minimal norm.
%We will see later,
%in Section~\ref{sec:OtherProofs},
%however, that this
%is implicit in previous proofs of the lower majorant property
%for~$L^{2j/(2j-1)}$.
%We will recall later why
% in Section~\ref{complements}
%there are functions in~$L^{2j/(2j-1)}$
%that have nonnegative coefficients,
%but that do not factor in the way specified above.
%Our method makes the restriction on the support of~$\hat G$ obvious.
%That restriction
%It 
%too was implicit in previous proofs. It leads to interesting restrictions on the support of the coefficients of the minimal majorant.

As~$p \to 1^+$ through the set of 
%peculiar
special 
exponents, 
Corollary~\ref{th:supports}
%this 
imposes
progressively weaker restrictions on the frequencies of the
minimal majorants of a 
%given 
trigonometric 
poynomial~$f$.
%polynomial.
%Another way to phrase those restrictions is to consider 
%Consider
%the exact
%majorant,~$E_f$ say, 
%of~$f$;
%of the trigonometric polynomial~$f$; 
%then
%the set of frequencies of the minimal majorant
%of~$f$ in~$L^{2j/(2j-1)}$ is included in the set
%of frequencies of~$(E_f)^j(\overline{E_f})^{j-1}$.
%Hence that minimal majorant must be equal to the convolution
%of~$(E_f)^j(\overline{E_f})^{j-1}$ with something, but our methods
%only reveal the nature of that convolution factor in special cases.
\end{remark}

\begin{remark}
Instead of imposing condition~\eqref{eq:mod-eq}
on suitable functions~$g$,
%on functions~$g$ with nonnegative coefficients that vanish off the support of~$\hat f$,
we
%one
could,
as in~\cite{FoArxiv},
%elect to
maximize the left-hand side of equation~\eqref{eq:mod-eq}
over
suitable
%such
functions~$g$ in
a
%suitable
closed
sphere
%ball
in~$L^{2j}$.
%the closed unit ball of~$L^{2j}$.
%Again, rescaling suitably would yield the desired function~$G$.  
\end{remark}

\begin{remark}
One way to formulate the upper majorant property in~$L^{2j}$ is that if
%trigonometric polynomials~$f$ and~$g$
two functions~$f$ and~$g$ in~$L^1$
have the property that~$|\hat f| \le |\hat g|$,
then~$\|f\|_{2j} \le \|g\|_{2j}$
in the cases
%when~$\hat g \ge 0$.
where~$\hat g \ge 0$.
Theorem~\ref{th:Dual Norm Inequality}
yields a dual version of this, namely that if~$|\hat f| \le |\hat g|$, then~$\|f\|_{2j/(2j-1)} \le \|g\|_{2j/(2j-1)}$
%in the special cases
%in all cases
when~$f$ factors
%as~$(\overline H)^{j-1}H^j$,
as~$(\overline H)^{j-1}H^j$
%where~$f$ factors as~$H^j(\overline H)^{j-1}$
where~$H \in L^{2j}$ and~$\hat H \ge 0$.
%where~$\hat H \ge 0$ and~$H \in L^{2j}$.
\end{remark}

\begin{remark}
Corollary~\ref{th:ownmajorant} also follows easily from part~(ii) of Lemma~2
in~\cite{Ba}.
\end{remark}

\section{Partial majorants}
\label{sec:PartMaj}

In this section, we show that 
much
%most
%part 
of the pattern
%presented
in Theorem~\ref{th:trig-th}
persists for other values of~$p$ in the interval~$(1, \infty)$
%for
with 
a
different
%weaker
%modified
notion of majorant.
In the next section, we 
observe
%note
that
%explain why
%show that 
the two notions coincide for minimal majorants when~$p'$ is even,
and
we explain
why
%we
%recall how
%that 
a 
key
%desired norm
estimate
holds
%follows
in those
special
cases.

Given
%Let~$c = (c_n)_{n=-\infty}^\infty$ be 
a
bounded
%any
sequence~$c = (c_n)_{n=-\infty}^\infty$,
%sequence,
%that tends to~$0$ at~$\pm\infty$;
%and
regard it
%regard~$c$ 
as giving the coefficients
of some~$2\pi$-periodic distribution,~$\check c$ say,
which may or may not belong to~$L^p$.

\begin{definition}
\label{def:partial}
An integrable function~$F$ is a \textit{partial majorant}
of~$\check c$
if~$\hat F(n) \ge |c(n)|$ at all indices~$n$ where~$c(n) \ne 0$.
\end{definition}

%In particular,~$\hat g \ge 0$ on the support of~$c$.
If~$\hat F \ge 0$ off the support of~$c$ too,
%then 
call~$F$ a \textit{full majorant} of~$\check c$.
When~$\check c$
is integrable,
%and~$F$ are
%integrable functions,
%is an integrable function,
this coincides
with the 
%previous 
notion
of~$F$
%of
being a majorant of~$\check c$.
%
%For each sequence~$c$, t
The
%set of
functions
in~$L^p$
%in~$L^p$, if any,
that partially majorize~$\check c$
form
%is
a closed convex subset of~$L^p$. 
If 
it
%this set
is nonempty, 
and~$p \in (1, \infty)$,
then 
this
subset
%set
%it 
has a unique element of minimal~$L^p$ 
norm.
%norm,
%since the~$L^p$ norm is uniformly convex.
As noted earlier, the
%The
same comments
apply to the set of functions
in~$L^p$ that fully majorize~$\check c$. Since 
that
set is included in
%this 
%is a subset of
the set
of partial majorants of~$\check c$ in~$L^p$, the minimal norm
of full majorants cannot be smaller
than the minimal norm of partial majorants.

Recall that complex function-valued functions~$F$ factor as~$|F|\sgn(F)$, where~$\sgn(F)$ vanishes off the support of~$F$.  Also, when~$F \in L^p$, where~$1 < p < \infty$,
letting
\begin{equation}
\label{eq:FormG}
%\[
G = |F|^{p/p'}\sgn(F)
= |F|^{p-1}\sgn(F)
%\]
\end{equation}
%letting~$G
%= |F|^{p/p'}\sgn(F)
%= |F|^{p-1}\sgn(F)$
puts~$G$ in ~$L^{p'}$, and
then
\begin{equation}
\label{eq:FormF}
%\[
F = |G|^{p'/p}\sgn(G)
= |G|^{p'-1}\sgn(G).
%\]
\end{equation}
%makes~$F = |G|^{p'/p}\sgn(G)
%= |G|^{p'-1}\sgn(G)$.

%The use of the function~$\phi$ above comes from the classical
%duality proof~\cite{HaL}.
%The classical ancestor of Lemma~\ref{th:classical}
%is on page~311 of~\cite{HaL}.
%It
%concerns full majorants, and does not include the statement
%that~$\hat G = 0$ off the support of~$c$.
%The modern duality proof~\cite{DGLPQ} led
%leads 
%to our next lemma, but we use the classical method to prove it. 
%This will provide
%That provides
%a second proof of Lemma~\ref{th:classical}, with further information, and will be the main  tool
%idea 
%used in the rest of this paper.
%All but the last sentence
%in Lemma~\ref{th:classical} follows easily from it. In the next section,
%we will reprove it in a more elementary but longer way that
%yields all of Lemma~\ref{th:classical} too.

\begin{definition}
\label{def:infimum}
Given a nontrivial 
bounded
sequence~$c$, let~$R(c)$ be the set
of trigonometric polynomials~$g$
with 
nonnegative
coefficients 
%that are nonnegative,
that vanish off the support of~$c$,
%with~$\hat g$ vanishing off the support of~$c$, 
and for which
\[
%$
\sum_n |c(n)|\hat g(n) = 1.
%\sum_n |c(n)|\hat g(n) = 1
%$. 
\]
Given an 
exponent~$p$
%index~$p$
in the interval~$(1, \infty)$,
%When~$1 < p < \infty$, 
let
%Let
%Consider the real number
\begin{equation}
\label{eq:Kp(c)}
%\[
K_p(c) = \inf\{\|g\|_{p'}: g \in R(c)\}.
%\]
\end{equation}
\end{definition}

\begin{lemma}
\label{th:norm}
Let~$1 < p < \infty$.
A nontrivial distribution~$\check c$ has a partial majorant
in~$L^p$ if and only
if~$K_p(c) > 0$;
%if~$K_p(c) > 0$, and
%if~$K_p(c) > 0$.
%In that case,
the minimal~$L^p$ norm of partial majorants
of~$\check c$ is
then
equal to~$1/K_p(c)$.
The partial majorant
%,~$F$ say, 
of
minimal~$L^p$
norm
%minimal norm
is a rescaled copy
of~$|h|^{p'-1}\sgn(h)$ for the function~$h$ of minimal~$L^{p'}$ norm
in the closure of the set~$R(c)$ in~$L^{p'}$.
Finally,~$\hat h$
%Finally,
%the transform of~$h$
%Finally,~$\hat h$
vanishes off the support of~$c$,
and on the part of
%and~$\hat h(n) = 0$ at all indices~$n$ in 
that support 
where the transform of the minimal majorant is strictly larger than~$|c|$.
%where~$\hat F > |c|$.
%where~$\hat F(n) > |c(n)|$.
\end{lemma}

\begin{corollary}
\label{th:classical}
Let~$1 < p < \infty$.
When~$\check c$ has a partial majorant in~$L^p$,
let~$F$ be its minimal partial majorant in~$L^p$,
and let~$G = |F|^{p-1}\sgn(F)$.
%there is a function~$G$ in~$L^{p'}$ for which
%\begin{equation}
%\label{eq:power}
%F = |G|^{p'-1}\sgn(G).
%\end{equation}
%then the partial majorant,~$F$ say, of~$\check c$ with minimal~$L^p$ norm has the 
%form~$|G|^{p'-1}\sgn(G)$
%form~$G|G|^{p'-2}$, 
%where~$G \in L^{p'}$
Then~$\hat G \ge 0$, and~$\hat G$ vanishes off the support of~$c$.
Moreover,~$\hat G$ also vanishes on the part of
that support
%the support of~$c$
where~$\hat F > |c|$.
\end{corollary}

\begin{proof}[Proofs]
%Let~$g \in R(c)$.
Suppose that~$\check c$ has a partial 
majorant,~$E$ say,
%majorant,~$F$ say,
in~$L^p$.
For functions~$g$ in~$R(c)$,
%If~$g \in R(c)$, then
%and let~$F$ be its partial majorant of minimal~$L^p$ norm.
%Then
only finitely-many of the terms in the sum~$\sum_n \hat g(n)|c(n)|$ are nontrivial,
%and~$\hat g(n)|c(n)| \le g(n)\overline{\hat F(n)}$ in those cases.
and then
\[
\hat g(n)|c(n)|
\le \hat g(n)\hat E(n)
=
\hat g(n)\overline{\hat E(n)}.  
\]
%and then~$\hat g(n)|c(n)|
%\le \hat g(n)\hat K(n)
%=
%\hat g(n)\overline{\hat K(n)}$.  
Therefore,
%So
\begin{equation}
%\begin{gather}
\label{eq:extremal}
\begin{gathered}
%\begin{split}
%1 = 
\sum_n \hat g(n)|c(n)| 
\le \sum_n g(n)\overline{\hat E(n)}
\\
%\notag
= \frac{1}{2\pi}\int g(\theta)\overline{E(\theta)}\,d\theta
\le \|g\|_{p'}\|E\|_p
\end{gathered}
%\end{split}
%\end{gather}
\end{equation}
by H\"older's inequality.
The assumption that~$\sum_n\hat g(n)|c(n)| = 1$ %then
yields that~$\|g\|_{p'} \ge 1/\|E\|_p$.
%So~$1/\|F\|_p$ is a lower bound
%for~$\|g\|_{p'}$ for all functions~$g$
%in the set~$R(c)$.
Hence~$K_p(c) \ge 1/\|E\|_p$;
%is also bounded below by the reciprocal of~$1/\|F\|_p$;
%by~$1/\|F\|_p$;
in particular,~$K_p(c) > 0$,
%in particular,~$K_p(c)$ is strictly positive,
as
asserted
%claimed
in the lemma.

For the converse, suppose that~$K_p(c) > 0$.
The
%Note that the 
set~$R(c)$
is nonempty since it contains the
function~$t \mapsto e^{int}/|c(n)|$
%monomial~$z_n/|c(n)|$
for each index~$n$ in the support of~$c$.
Because~$R(c)$
%Since~$R(c)$
is convex, 
and~$1 < p' < \infty$,
%its closure
the
%The 
closure of~$R(c)$ 
in~$L^{p'}$
then
%is then nonempty and convex, and so 
contains
an
%a unique 
element,~$h$ say, of minimal~$L^{p'}$ norm.
That norm must
be~$K_p(c)$;
moreover,~$\hat h$ is nonnegative, and~$\hat h$ vanishes off the support of~$c$.
%be~$K_p(c)$.

Let
\[
J
%=  h(t)|h(t)|^{p'-2} 
=  |h|^{p'-1}\sgn(h)
= |h|^{p'/p}\sgn(h).
\]
%and note that
Then~$\|J\|_p = 1$ 
%in the special case where
when~$\|h\|_{p'}=1$,
that is when~$K_p(c) = 1$.
%and let~$J = k/K_p(c)^{p'}$.
%Then
%\begin{gather*}
%\[
%(\|k\|_p)^p = (\|h\|_{p'})^{p'} = K_p(c)^{p'},
%\quad\textnormal{and}\\
%\]
%Let ~$J = k/K_p(c)^{p'}$. Then
%\[
%\textnormal{and}\quad
%\|J\|_p = K_p(c)^{(p'/p)-p'}
%= K_p(c)^{-p'(1 - 1/p)}= \frac{1}{K_p(c)}.
%\]
%\end{gather*}
We claim 
that~$J$ is a partial majorant of~$\check c$
in
this
%that
special
case.
%when~$K_p(c) = 1$.
%in that case.
%that~$J$ 
%is a partial majorant of~$\check c$.

If so, then~$J$
%then 
%it 
must be the partial 
majorant
%majorant~$F$ 
of minimal~$L^p$ norm,
since the
discussion
%inequality
after the relations~\eqref{eq:extremal} makes~$1/K_p(c)$ a lower bound
for the~$L^p$ norms of all partial majorants of~$\check c$.
%
%To verify the claim, first use the fact that~$h$ is in
%the closure of~$R(c)$ to see that~$\hat h(n) \ge 0$
%for all~$n$, and that~$\hat h(n) = 0$ for all~$n$
%outside the support of~$c$.
Since~$\hat h$
%In particular~$\hat h$
is real-valued,~$h(-\theta) = \overline{h(\theta)}$
%so that~$h(-\theta) = \overline{h(\theta)}$
%and this makes~$h(-t) = \overline{h(t)}$
for almost all~$\theta$. It follows
that~$J(-\theta) = \overline{J(\theta)}$
for almost all~$\theta$, and hence that~$\hat J$ is
real-valued.
%real-valued too.

Fix an integer~$n$ in the support of~$c$. 
Let~$z_n$ be the function 
mapping~$\theta$ to~$e^{in\theta}$,
and let~$\phi$
%Let~$\phi$
%This time, let~$\phi$
%Let~$\phi$
map
%be the function mapping
real numbers~$r$
to~$(\|h + rz_n\|_{p'})^{p'}$;
%again,
%then, 
as in~\cite[Lemma~2]{HaL},
but with~$p$ replaced by~$p'$,
%then%~$
\begin{equation}
\label{eq:HLderivative}
%\[
\phi'(0) = p'\Real\hat J(n) = p'\hat J(n).
%\]
\end{equation}
%$.
%When~$p' = 2j$, proving this is elementary, because
%then~$\phi(r)$ is a polynomial of degree~$2j$.

If~$g \in R(c)$, and~$r>0$,
then the function~$(g + rz_n)/(1 + r|c(n)|)$
also
belongs to the set~$R(c)$.
%when~$r$ is a positive constant.
So the quotient
\begin{equation}
\label{eq:ratio}
\frac{h+rz_n}{1+r|c(n)|}
\end{equation}
belongs to the closure of~$R(c)$ when~$r>0$.
By the minimality of~$\|h\|_{p'}$
in~$R(c)$,
the derivative
at~$r=0$
%of~$(\|h+rz_n\|_{p'})^{p'}/(1 + r|c(n)|)^{p'}$
of
\[
\frac{(\|h+rz_n\|_{p'})^{p'}}{(1 + r|c(n)|)^{p'}}
\]
%the~$p'$-th power of the~$L^{p'}$ norm
%of that ratio
must be nonnegative. That derivative is equal to
\begin{equation}
\label{eq:derivative}
\frac{p'\hat J(n)[1+0]^{p'}
- (\|h+0\|_{p'})^{p'}p'[1+0]^{p'-1}|c(n)|}
{[1+0]^{2p'}}.
\end{equation}
This is nonnegative
for all~$n$  in the support of~$c$
if and only if
\[
\hat J(n) \ge |c(n)|
%\hat J(n) \ge |c(n)|,
%\hat k(n) \ge K_p(c)^{p'}|c(n)|,
\]
%and that makes~$\hat J(n) \ge |c(n)|$
in all those cases.
%for all~$n$  in the support of~$c$.
%That is,~$J$ is a partial majorant for~$\check c$.
So~$J$
%So~$J$ 
is indeed a partial majorant of~$\check c$.

When~$\hat h(n) > 0$, the quotient~\eqref{eq:ratio} also belongs
to the closure of the set~$R(c)$ when~$r$ is negative
and close enough to~$0$.
In those cases,
the derivative~\eqref{eq:derivative} must
%again
be equal to~$0$, and
then~$
\hat J(n)
%\hat F(n) = \hat J(n) 
= |c(n)|$.
%Equivalently,
%
%This analysis using the function~$z_n$ applies when~$n$ belongs to the support of~$c$.
So
%It follows that
the part of
%that
the
support
of~$c$
where~$\hat J(n) > |c(n)|$ is disjoint from the support
of~$\hat h$,
and~$\hat h$ must vanish on that part. 
%of~$\hat h$.  
%if~$\hat k(n) > |c(n)|$.
%$ > 0$.

If~$K_p(c)$ is positive but differs from~$1$, let~$c' =  c/K_p(c)$, and note that~$K_p(c') = 1$.
Let~$h' = K_p(c)h$; this is
%Take 
the function of minimal~$L^{p'}$ norm in the set~$R(c')$.
%Let~$J' = |h'|^{p'-1}\sgn(h')$.
%and form the corresponding function~$k'$.  
Then~$\|h'\|_{p'} = 1$,
and the function~$J'$ that factors
as~$|h'|^{p'-1}\sgn(h')$
%and~$J'$
is the minimal
partial
majorant 
in~$L^p$
of the inverse transform of~$c'$.
Moreover~$\widehat{h'}(n) = 0$
on the part of the support of~$c'$ where~$\widehat{J'}(n) > |c'(n)|$.
%if~$\widehat{J'}(n) > |c'(n)|$.
Let
\[
F = K_p(c)J' = K_p(c)^{p'}J;
\]
%Let~$F = K_p(c)J' = K_p(c)^{p'}J$;
this is the
minimal
partial
majorant of~$\check c$ in~$L^p$, and the rest of the lemma follows.
%Let~$G = K_p(c)^{1/(p'-1)}$
%and~$F(t) = |G(t)|^{p'-1}\sgn[G(t)]$. 
Corollary~\ref{th:classical} also
follows,
%follows
%Lemma~\ref{th:classical} also follows
%and~$G = K_p(c)^{p/p'}h' = K_p(c)^ph$.
%\end{proof}
and
\[
G = K_p(c)^{p/p'}h' = K_p(c)^ph.
\qedhere
\]
\end{proof}

\begin{remark}
\label{rm:partialinstances}
The classical method was applied to some instances of partial majorants
and related notions
%in~\cite{HaL}.
in~\cite[page 308]{HaL}.
%in~\cite{HaL} and~\cite[Lemma~2]{Ba}.
It will be used
%We explain
in Section~\ref{sec:OtherProofs}
to give another proof of
Corollary~\ref{th:classical}.
%how Corollary~\ref{th:classical} follows easily by that method.
\end{remark}

\section{Special exponents}\label{sec:SpecialInd}

%Except where we say otherwise, we suppose in the rest of this paper that~$p'$ is even and strictly greater than~$2$.
The lower majorant property
for~$L^p$
holds in 
the 
%peculiar
%those 
special 
cases
because
of
two
%some
extra
things that are true
%facts about
%properties of
%minimal
%partial 
%majorants
in those cases.
%majorants.
%follows easily from the following
%two 
%special facts about the notions in the previous section.

\begin{lemma}
\label{th:special}
Let~$p = 2j/(2j-1)$
%Let~$p' = 2j$
for some 
%positive 
integer~$j>1$,
and let~$c$ be a 
%nontrivial
bounded 
sequence.
If~$\check c$ has a partial majorant in~$L^p$,
then its partial majorant of minimal~$L^p$ norm
is also a full majorant
of minimal~$L^p$ norm.
If there is a 
nontrivial
function~$f$ in~$L^p$ for which~$|c| \le |\hat f|$,
%for some~$f$ in~$L^p$,
%and~$\|f\|_p > 0$,
then~$K_p(c) \ge 1/\|f\|_p$.

\end{lemma}

\begin{proof}
For these special values of~$p$,
if~$\check c$ has a partial majorant in~$L^p$, then
the
%The
%Since~$p'-2 = 2(j-1)$, the
factorization of the minimal
partial majorant~$F$
%specified
in Lemma~\ref{th:classical} takes the form
\begin{equation}
\label{eq:special}
%\[
F =
|G|^{p'-2}|G|\sgn(G)
%|G|\sgn(G)|G|^{p'-2}
= \left\{|G|^2\right\}^{j-1}G = \left\{\overline GG\right\}^{j-1}G
= \overline G^{j-1}G^j,
%\]
\end{equation}
where~$G \in L^{2j}$ and~$\hat G \ge 0$.
%Suppose that~$\check c$ has a partial majorant in~$L^p$. Divide each member of~$R(c)$ by~$K_p(c)^p$ to get the set~$R(c)/K_p(c)^p$ whose~$L^{p'}$ closure contains~$G$.
Choose 
trigonometric polynomials~$G_m$ with nonnegative coefficients
%functions~$G_m$ in this set
so that the sequence~$\{G_m\}_{m=1}^\infty$ converges in~$L^{p'}$ norm to~$G$.
Then the sequence~$\{(\overline{G_m})^{j-1}(G_m)^j\}_{m=1}^\infty$
converges in~$L^p$ norm to~$F$.
Since the coefficients of
each function~$G_m$
%each~$G_m$
are nonnegative, the same is true for
%the coefficients of
the terms
in the sequence~$\{(\overline{G_m})^{j-1}(G_m)^j\}_{m=1}^\infty$,
and hence for
the~$L^p$-norm
%the norm
limit~$F$ of that sequence.
So the partial majorant~$F$
%So~$F$
is a full majorant of~$\check c$.
Its~$L^p$ norm 
is
%must be 
minimal, since~$\|F\|_p$ is minimal among norms of partial majorants of~$\check c$.

%Moreover,~$\hat h = 0$ off the support of~$c$.
%Since~$G$ and~$F$ come from~$h$ and~$k$

Suppose 
%~$c$ is nontrivial, and
that~$|c| \le |\hat f|$ 
for some~$f$ in~$L^p$;
then there is sequence~$(\varepsilon(n))$ of numbers
of absolute-value at most~$1$
for which~$|c(n)|
%so that~$|c(n)|
= \varepsilon(n)\overline{\hat f(n)}$
for all indices~$n$.
%When~$c$ is nontrivial,
Let~$g \in R(c)$, and let~$k$ be the trigonometric
polynomial
for which~$\hat k(n) = \varepsilon(n)\hat g(n)$
%with~$\hat k(n) = \varepsilon(n)\hat g(n)$
for all~$n$.
Much as
in the relations~\eqref{eq:extremal},
%As
%in formula~\eqref{eq:extremal},
%Then
%\begin{equation}\label{parseval}
\begin{equation}
%\begin{gather}
%\begin{split}
\label{eq:Switch}
\begin{gathered}
%\[
%1 = 
\sum_n \hat g(n)|c(n)|
= \sum_n \hat g(n)\varepsilon(n)\overline{\hat f(n)}\\
%\]
%\[
= \sum_n \hat k(n)\overline{\hat f(n)}
= \frac{1}{2\pi}\int_{-\pi}^\pi
k(\theta)\overline{f(\theta)}\,d\theta.
%\end{split}
%\end{gather}
%\]
\end{gathered}
\end{equation}
By H\"older's inequality
%, the fact that~$k$ is a minorant of~$g$,
and the upper majorant property 
with constant~$1$
in~$L^{2j}$,
\begin{equation}
\label{eq:minorant}
\left|\frac{1}{2\pi}\int_{-\pi}^\pi
k(\theta)\overline{f(\theta)}\,d\theta\right|
\le \|k\|_{2j}\|f\|_p \le \|g\|_{2j}\|f\|_p.
\end{equation}
The
assumption that~$\sum_n\hat g(n)|c(n)| = 1$
%now
%then
yields that~$\|g\|_{2j} \ge 1/\|f\|_p$
%Thus~$1 \le \|g\|_{p'}\|f\|_p$
%and~$\|g\|_{p'} \ge 1/\|f\|_p$
for all functions~$g$
in~$R(c)$,
%in~$R(c)$.
and that~$K_p(c) \ge 1/\|f\|_p$.
%Therefore~$K_p(c) \ge 1/\|f\|_p$.
%So~$K_p(c) \ge 1/\|f\|_p$.
%It follows that~$K_p(c) \ge 1/\|f\|_p$.
%as claimed.
\end{proof}

\begin{proof}[Proof of Theorem \ref{th:trig-th}]
Let~$p = 2j/(2j-1)$, and let~$f$ be a nontrivial function in~$L^p$.
By the second part of the lemma above,
\[
K_p\left(\hat f\right) \ge \frac{1}{\|f\|_p} > 0.
\]
By
the first part of
Lemma~\ref{th:norm}, the function~$f$ has a partial majorant
in~$L^p$, and its partial majorant~$F$ of minimal~$L^p$ norm
satisfies the condition that
\begin{equation}
\label{eq:Norm}
\|F\|_p
=
%\le
\frac{1}{K_p
%\left
(\hat f
%\right
)} \le \|f\|_p.
\end{equation}
By Lemma~\ref{th:special}, the minimal partial majorant~$F$
is also a full majorant,
with
%of
minimal~$L^p$ norm, 
of~$f$.
%for~$f$.
%Lemma~\ref{th:classical} 
%then
%gives 
%the factorization
%claimed in Theorem~\ref{th:trig-th}.
% except for the estimate
%on~$\|G^j(\overline G)^{j-1}\|_{p}$, which 
%follows
%holds 
%because
%\[
%(\|G\|_{p'})^{2j} = (\|F\|_p)^{2j/(2j-1)} \le (\|f\|_p)^{2j/(2j-1)}.
%\qedhere
%\]
\end{proof}

\begin{proof}[Proof of Corollary~\ref{th:supports}]
The conclusion is evident if~$G$ is
a trigonometric polynomial. In general,~$G$
is again a limit
%is a limit
in~$L^{2j}$ norm
%in~$L^{p'}$ norm
of a sequence~$(G_m)$
of polynomials
%~$G_m$ say,
whose coefficients vanish off the support of~$\hat f$.
%in the set~$R(c)/K_p(c)^p$. 
%Then the minimal majorant of~$f$ is a limit in~$L^{p}$ norm of the sequence
The
%Then
%the
%the set of
frequencies of~$(\overline G)^{j-1}G^j$
have
%has
the desired property,
because
the frequencies of
the terms in
the sequence~$\left\{(\overline{G_m})^{j-1}(G_m)^j\right\}_{m=1}^\infty$ 
do.
%do,
%and that sequence
%the sequence with terms~$(\overline{G_m})^{j-1}(G_m)^j$
%converges
%again
%in~$L^p$ norm to~$(\overline G)^{j-1}G^j$.
%as~$m \to \infty$.
\end{proof}

\begin{remark}
\label{rm:VanishOnSupportOfG}
%Just
%Indeed, just 
%modify the sequence~$\varepsilon$ to vanish off 
%that set
%the support of~$g$.
%and note that 
Since
the products in the sums~\eqref{eq:Switch}
all vanish off
the support of~$\hat g$,
%factoring~$|c(n)| $ as~$\varepsilon(n)\overline{\hat f(n)}$ and~$\hat k(n)$ as~$\varepsilon(n)\hat g(n)$,
%the
requiring that~$|c(n)| = \varepsilon(n)
%requirement that~$|c(n)| = \varepsilon(n)
%factorizations~$|c(n)| = \varepsilon(n)
\overline{\hat f(n)}$
and~$\hat k(n) = \varepsilon(n)\hat g(n)$,
%it
%suffices
%for those equations
%that~$|c(n)| = \varepsilon(n)\overline{\hat f(n)}$
%is enough for the equations~$|c(n)| = \varepsilon(n)\overline{\hat f(n)}$
%and~$\hat k(n) = \varepsilon(n)\hat g(n)$,
%to hold,
where~$|\varepsilon(n)| \le 1$,
%with~$|\varepsilon(n)| \le 1$,
is only needed
%are only necessary
%in the relations
when~$\hat g(n) \ne 0$.
%Choose~$\varepsilon(n)$ as before in those case,
%and let~$\varepsilon(n)$ vanish otherwise.
%So
The conclusion
%in the relations~\eqref{eq:Switch}
that
\begin{equation}
\label{eq:KeyInequality}
\sum_n \hat g(n)|c(n)|
\le \|g\|_{2j}\|f\|_p
%\le \|g\|_{p'}\|f\|_p
\end{equation}
%that~$\sum_n \hat g(n)|c(n)| \le \|g\|_{p'}\|f\|_p$
%that~$\|g\|_{p'} \ge 1/\|f\|_p$
%The lower bound on~$\|g\|_{p'}$ given in the last paragraph of the proof of Lemma~\ref{th:special}
therefore
%also 
follows from the weaker hypothesis
that~$|c| \le |\hat f|$ 
on the support of~$\hat g$.
Moreover, the assumption that~$\|g\|_{2j} = 1$
%was set equal to~$1$, but that restriction 
was not used for this.
%is not needed for this, because it
%was not used
%until later, in getting the lower bound for~$K_p(c)$.
These observations will be useful
in
the first part of
Section~\ref{sec:complements}.
\end{remark}

\begin{remark}
\label{rm:NoAccident}
%We use special properties
%of~$L^{p'}$ norms when~$p'$ is even.
%rather than covering more general norms.
%, such as the upper majorant property with constant~$1$,
%These properties will also be available
%in Section~\ref{Continuous}, where we study these questions
%on~$\R^m$.
%Our proof does not rely nearly as much
%on facts about summability as previous proofs did.
%
%For instance, t
The cases where~$p'=2j$ 
for a positive integer~$j$
are the only ones~\cite{Ru}
where the function~$\sgn(G)|G|^{p'-1}$ must have nonnegative coefficients
if~$G$ does. Hence the part of Lemma~\ref{th:special}
making~$\hat F$ nonnegative fails when~$p'$ is not even.
So does the other part, using the upper majorant property
for~$L^{p'}$ to get a strictly positive lower bound for~$K_p(\hat f)$.
This is no accident, since one way~\cite{Shap, LeS, Ra} to disprove the upper
majorant property when~$p'$ is not even is to use a function~$g$ with nonnegative coefficients for which~$\sgn(g)|g|^{p'-1}$ has a negative coefficient.
\end{remark}

\begin{remark}
\label{rm:polynomial}
Proving Lemma~\ref{th:special} only requires Lemma~\ref{th:norm} in the cases where~$p'$ is even.
%Then
%The proof of the latter is then more elementary, because 
%the function~$\phi$ in
The
%Then
%the
proof
of the latter
lemma
%is a
is
more elementary
%simpler
in those cases,
%simpler,
because the function~$\phi$
%because~$\phi(t)$
in formula~\eqref{eq:HLderivative}
is
then
a
polynomial.
%polynomial in~$t$.
%polynomial, and proving
\end{remark}

\section{Previous proofs
% and supports
}
\label{sec:OtherProofs}

%Now we comment on previous proofs of Theorem~\ref{th:trig-th}.
The
%As noted earlier, the
classical version~\cite{HaL} of 
Corollary~\ref{th:classical}
%Lemma~\ref{th:norm}
applies to full majorants of minimal norm. It does not include the statement that~$\hat G \equiv 0$ off the support of~$c$, but it asserts when~$\check c \in L^p$ that
\begin{equation}
\label{eq:slack}
\hat G(n) = 0 \text{ if } \hat F(n) > |c(n)|
\end{equation}
whether or not~$n$ lies in the support of~$c$.
%It is proved in the same way as Lemma~\ref{th:classical}.

We now
%explain why
show how
%that 
the conclusion
in Theorem\ref{th:trig-th}
that~$\hat G \equiv 0$
off the support of~$c$ follows from condition~\eqref{eq:slack}
in the special cases
where~$p'$ is even.
%where~$p' = 2j$ for some integer~$j$.
%because Property~\eqref{en:support} in Definition~\ref{def:conjugate} follows from the other properties of~$p$-conjugates in those cases.
%Then the form
%of the factorization simplifies to~$F = G^j\overline G^{j-1}$.
%Consider the functions~$H$ and~$J$ 
%There is nothing to prove
%This is trivial 
%if~$\hat G \equiv 0$.
In
nontrivial
%the remaining
cases,
%When~$\hat G$ is nontrivial, 
view the coefficients of~$F$ as being
given by the convolution of the coefficients of~$G$ with those
of~$(|G|^2)^{j-1}$.
Now~$\widehat{\overline G}(n) = \overline{\hat G(-n)} = \hat G(-n)$
for all~$n$. So the coefficients of~$\overline G$ are all nonnegative,
and the same is true for the coefficients  of~$|G|^2$, which come from
convolving the coefficient sequences of~$G$ and~$\overline G$.
Moreover,
\[(|G|^2)\widehat{\rm\ \ }(0)
= \sum_n \hat G(n)
\widehat{
\overline G
}
%\widehat{\rm\ \ }
(-n)
= \sum_n \hat G(n)^2 > 0.\]
By induction on~$j$,
%Similarly,
the coefficients
of~$|G|^{2j-2}$
%of~$|G|^{2j-2} = (|G|^2)^{j-1}$
are all
nonnegative,
and its~$0$-th coefficient is {strictly} positive. Therefore
\begin{equation}
\label{eq:eFFhat}
\hat F(n) = \sum_m \hat G(n-m)(|G|^{2j-2})\widehat{\rm\ \ }(m)
\ge \hat G(n)\widehat{|G|^{2j-2}}
%)\widehat{\rm\ \ }
(0),
\end{equation}
and the last expression is strictly positive if~$\hat G(n) > 0$.
Hence the support of~$\hat G$ is included in the support of~$\hat F$.

The support of~$\hat G$
%It
is
%will
also included in the support of~$c$
%if~$\hat G \equiv 0$
if~$\hat G$
%also
vanishes
on the part of the support of~$\hat F$
lying outside the support of~$c$.
If~$n$ belongs to
%Moreover, on
%On
that part,
%the part
%of
%the latter
%that
%set
%support
%lying outside the support of~$c$,
then~$\hat F(n) > 0 = |c(n)|$.
%\[
%\hat F(n) > 0 = |c(n)|.
%\]
%Condition~\eqref{eq:slack}
Since condition~\eqref{eq:slack}
%Using condition~\eqref{eq:slack}
%then
%yields that~$\hat G(n) = 0$.
makes~$\hat G(n) = 0$ in that case,
%This is not an effective restriction since we do not know~$F$.
%It follows,
%We claim, 
%however, that the support of~$\hat G$ is also included in
%the support of~$c$,
%which we do know. Indeed, suppose to the contrary
%that~$c(n) = 0$ but~$\hat G(n) \ne 0$. Then~$\hat G(n) > 0$, and
%it follows from the previous paragraph that
%$\hat F(n) > 0 = |c(n)|$. Using condition~\eqref{eq:slack}
%then yields that~$\hat G(n) = 0$,
%contrary to our hypothesis that~$\hat G(n) \ne 0$.
%So
the support
of~$\hat G$ is
indeed
included in the
%known
support of~$c$.

Another way to reach this conclusion is to prove
Corollary~\ref{th:classical}
%by the classical method.
by the classical method,
as was done in special cases
%in special cases 
in~\cite{HaL}.
%in~\cite{HaL} and~\cite[Lemma~2]{Ba}.
To that end, let~$p\in(1,~ \infty)$,
let~$F$ be a minimal partial majorant of~$\check c$ in~$L^p$,
and let~$G = |F|^{p-1}\sgn(F)$.
%partial majorants.
%It yields more information about them when~$p'$ is not even than it does for minimal majorants in that case.
%
%\begin{lemma}
%\label{th:classical}
%Let~$1 < p < \infty$. When~$\check c$ has a partial majorant in~$L^p$, let~$F$ be its minimal partial majorant in~$L^p$, and let~$G = |F|^{p-1}\sgn(F)$.
%there is a function~$G$ in~$L^{p'}$ for which
%\begin{equation}
%\label{eq:power}
%F = |G|^{p'-1}\sgn(G).
%\end{equation}
%then the partial majorant,~$F$ say, of~$\check c$ with minimal~$L^p$ norm has the 
%form~$|G|^{p'-1}\sgn(G)$
%form~$G|G|^{p'-2}$, 
%where~$G \in L^{p'}$
%Then~$\hat G \ge 0$, and~$\hat G$ vanishes off the support of~$c$. Moreover,~$\hat G$ also vanishes on the part of the support of~$c$ where~$\hat F > |c|$.
%\end{lemma}
%
%\begin{proof}
%^When~$p' < 2$, the factorization specified in the lemma has the form~$0/0$ at points where~$G(t) = 0$. Deal with this ambiguity by writing
%\begin{equation}
%\label{eq:power}
%F = |G|^{p'-1}\sgn(G),
%\end{equation}
%where~$|\sgn(G)| = 1$ and~$G = |G|\sgn(G)$. 
%Equation~\eqref{eq:power} is true if and only if
%Make equation~\eqref{eq:power} true by letting
%\begin{equation}
%\label{eq:inverse}
%G = |F|^{1/(p'-1)}\sgn(F).
%\end{equation}
%This is also equal to~$|F|^{p-1}\sgn(F)$, since~$(p-1)(p'-1) =1$.
%
The
%note
%Note
%first
%that the
function~$\tilde F$ mapping~$\theta$ to~$\overline{F(-\theta)}$ has the same~$L^{p}$ norm as~$F$, and~$\hat{\tilde F} = \overline{\hat F}$.
So~$\tilde F$ is also a minimal partial majorant of~$\check c$ in~$L^{p}$,
and must
therefore
coincide almost everywhere with~$F$ 
by the minimality of~$\|F\|_p$.
It follows that~$\tilde G$ coincides almost everywhere with~$G$, and hence that~$\hat G$ is real-valued too.

%To deduce Corollary
Fix an integer~$n$
outside
the support of the sequence~$c$.
When~$r$ is real, let~$\psi(r) = (\|F + rz_n\|_p)^p$,
where~$z_n$ is again the function 
mapping~$\theta$ to~$e^{in\theta}$.
%mapping~$t$ to~$e^{int}$.
As in~\cite[page~311]{HaL},
the derivative~$\psi'(0)$ exists and is equal to
the real part of~$p\hat G(n)$,
which is just~$p\hat G(n)$ here.
%When~$r \ge 0$,
%Then~$\psi(r)$ is minimal at~$r=0$, since
%Then
%Since
The
%the
functions~$F + rz_n$
are 
also
partial 
%majorants of~$\check c$,
majorants of~$\check c$.
%and~$F$ is the minimal such majorant.
%As in~\cite[page~311]{HaL},
%the derivative~$\psi'(0)$ exists and is equal to
%the real part of~$p\hat G(n)$,
%which is just~$p\hat G(n)$ here.
By the minimality of~$\|F\|_p$
among partial majorants,~$\psi'(0) = 0$.
%among partial majorants,
%By the minimality of~$\|F\|_p$,
%that derivative must be equal to~$0$.
Hence~$\hat G(n) = 0$ for all~$n$ outside the support of~$c$.

Argue similarly when~$n$ lies in the support of~$c$,
and~$\hat F(n) > |c(n)|$. At the remaining points~$n$
in that support,~$\hat F(n) = |c(n)|$, and
the functions~$F + rz_n$ with~$r>0$
are partial majorants of~$\check c$.
So~$\psi'(0) \ge 0$,
and~$\hat G(n) \ge 0$ at these points.
%\end{proof}

%We now outline a proof of Theorem~\ref{th:trig-th} based on an analysis
%The authors of the modern duality proof~\cite{DGLPQ}
%specialized to the context
%of the peculiar~$L^p$ spaces and their duals.
%Our Lemma~\ref{th:norm} uses those ideas, but
%with the following differences.
The
%modern
proof
in~\cite{DGLPQ}
was also applied
%was applied
%also applies
to full majorants rather than partial majorants. 
%We describe it for~$L^p$ norms, but it also works for some other norms.
%Their method has the advantage of applying to 
%some norms that are not~$L^p$ norms.
%but only the latter are considered here.
%They
%also suppose initially that~$f$ is 
Given a nontrivial
function~$f$ in~$L^p$,
%trigonometric polynomial~$f$,
%, and they get the other cases
%later by an approximation argument.
%So the existence
%of a majorant is not an issue initially, but they need a version
%of the norm estimate~\eqref{eq:Norm}. Essentially, they get
%that estimate by considering
%and
%the authors of~\cite{DGLPQ}
form the set,~$\widetilde{
%\tilde 
R_{p'}}(\hat f)$ say,
%defined like our set~$R(\hat f)$
%but 
%containing
of 
all functions~$g$ in~$L^{p'}$
where
%for which
the coefficients of~$g$ are nonnegative 
and
\begin{equation}
\label{eq:constant1}
%\[
%$
\sum_n \hat g(n)|\hat f(n)| \ge 1.
%\]
\end{equation}
This 
set
is larger than the closure 
of~$R(\hat f)$
%of our set~$R(\hat f)$
in~$L^{p'}$,
%since~$\sum_n \hat g(n)|\hat f(n)|$
since the sum 
above 
is allowed to exceed~$1$, and there is no restriction on the support of~$\hat g$.
When~$p'$ is even, 
however,
the upper majorant property with constant~$1$ makes the minimum
value of~$\|g\|_{p'}$ on~$\widetilde{R_{p'}}(\hat f)$
the same as its minimum value on the closure of~$R(\hat f)$ in~$L^{p'}$.

The authors of~\cite{DGLPQ} use
continuous linear functionals
%a continuous linear functional 
whose real parts separate
%whose real part separates
the set~$\widetilde{R_{2j}}(\hat f)$ from
%the set~$\widetilde{R_{p'}}(\hat f)$ from
closed balls of radius less than~$K_p(\hat f)$
%the open ball of radius~$K_p(\hat f)$
about~$0$. Representing
those functionals
%that functional
by integration against
functions
%a function
%function~$F$ 
in~$L^p$,
and rescaling suitably,
yielded
%yields
the desired majorant for~$f$.
It is then only a short step to consider the function~$g$ that achieves the minimum norm in
the set~$\widetilde{R_{2j}}(\hat f)$,
%the set~$\widetilde{R_{p'}}(\hat f)$,
form~$(\overline g)^{j-1}g^j$,
%form~$g(\overline gg)^{j-1}$,
and proceed as we do.

\section{Complements}
%\section{
%Three
%Other 
%complements.}
\label{sec:complements}

%We continue to suppose that~$p'$ is an even integer.
%For an integer~$j >1$, 
%let~$p=2j/(2j-1)$,
%and
%let~$f$,~$F$ and~$G$ be the functions
%mentioned in Theorem~\ref{th:trig-th}.
%and denote the support of~$\hat f$ by~$S_f$.
%Since~$f$ has a majorant with no larger~$L^p$ norm, and~$F$ is its minimal majorant,~$\|F\|_p \le \|f\|_p$.
%by definition,~$\|F\|_p \le \|K\|_p$
%moreover,~$\|F\|_p \le \|K\|_p$
%for all other majorants of~$f$. 
%
%Slightly weakening the hypotheses used in
%proofs
%the classical proof
%of
%to prove 
%this
%the first
%that 
%inequality allows us to 
%confirm the uniqueness of~$p$-conjugates, and to identify the functions in~$L^p$ that are minimal majorants in~$L^p$ of something.

%\begin
%{theorem}
%{lemma}
%Let~$G$ be a function in~$L^{p'}$ with nonnegative coefficients,
%and let~$F = (\overline G)^{j-1}G^j$.
%Let~$f$ be any function in~$L^{p}$
%with the property that~$|\hat f| \ge \hat F$ on the support of~$\hat G$.
%Then~$\|F\|_{p} \le \|f\|_{p}$.
%\end
%{theorem}
%{lemma}

%\begin{corollary}
%\label{th:minimalconjugate}
%Let~$H$
%Let~$p=2j/(2j-1), and let~$H$ 
%be a~$p$-conjugate of a function~$f$ in~$L^p$.  
%Then
%Then~$(\overline H)^{j-1}H^j$ is 
%the minimal majorant of~$f$ in~$L^p$
%is~$(\overline H)^{j-1}H^j$, and~$H$ is the only~$p$-conjugate of~$f$.
%\end{corollary}

%\begin{corollary}
%\label{th:ownmajorant}
%If~$H \in L^{p'}$ and~$\hat H \ge 0$,
%then~$(\overline H)^{j-1}H^j$ is its own minimal majorant in~$L^{p}$.
%\end{corollary}

We now prove Theorem~\ref{th:Dual Norm Inequality} and its
corollaries.
%corollaries;
Then we apply them to related questions.
%then we apply them as promised.
%
%\begin{proof}[Proofs]
%of Theorem~\ref{th:Dual Norm Inequality} and its corollaries]
%In the calculation~\eqref{eq:Switch}, we used that fact that
%\begin{equation}
%\label{eq:ReParseval}
%\[
%\frac{1}{2\pi}\int_{-\pi}^\pi
%k(\theta)\overline{f(\theta)}\,d\theta
%= \sum_n \hat k(n)\overline{\hat f(n)}
%\]
%\end{equation}
%when~$k$ is a trigonometric polynomial and~$f \in L^p$.
%Standard methods in~\cite{HaL} or~\cite{Ba},
%show that this extends to all cases where~$1 < p < \infty$ and ~$k \in L^{p'}$.  We only need it 
%when~$p'$ is even and~$k$ has a majorant in~$L^{p'}$.  Then the equality also follows from the fact~\cite{Ba} that, in those cases, the Fourier series of~$k$ converges unconditionally in~$L^{p'}$.  In Remark~\ref{rm:unconditional} below, we explain how that convergence is an elementary consequence of the upper majorant property.
%
%As usual,~$p = 2j/(2j-1)$ for some integer~$j$.

Because~$\|F\|_p = (\|G\|_{2j})^{2j-1}$,
%Since~$\|F\|_p = (\|G\|_{p'})^{2j-1}$,
the
%desired
conclusion
in the
theorem,
namely
that~$\|F\|_p \le \|f\|_p$,
%theorem
%that~$\|f\|_p \ge \|F\|_p$
%lemma
follows if~$(\|G\|_{2j})^{2j-1}
\le \|f\|_p$. 
%is equivalent to the
%inequality~$(\|G\|_{2j})^{2j-1}
%\le \|f\|_p$. 
%\le (\|f\|_p)^{1/(j-1)}$. 
%inequality~$\|G\|_p \le (\|f\|_p)^{1/(j-1)}$.  
Consider the set,~$S_G$ say, of trigonometric polynomials
that are majorized by~$G$ and have only nonnegative coefficients.  Since~$G$ belongs to the closure
in~$L^{2j}$
%in~$L^{p'}$
of this set, it is enough to show that
\begin{equation}
\label{eq:Powerofnormf}
(\|g\|_{2j})^{2j-1} \le \|f\|_p
%\|g\|_{p'} \le (\|f\|_p)^{1/(j-1)}
\quad
\text{for all~$g$ in~$S_G$.}
\end{equation}

For any such function~$g$,
%let~$F' = (g\overline g)^{j-1}g$.
let~$F' = (\overline g)^{j-1}g^j$.
Then~$F'$ is majorized by~$F$,
so that~$\widehat{F'} \le \hat F \le |\hat f|$
%and~$\widehat{F'} \le \hat F \le |\hat f|$
on the support of~$\hat G$, and hence on the support of~$\hat g$.
%When~$g$ is nontrivial, replace~$g$ by~$g/\|g\|_{p'}$ and~$f$ by~$f/\|g\|_{p'}^{j-1}$
%to reduce matters further to the case where~$\|g\|_{p'} = 1$.
%and the goal is to show that~$\|f\|_p \ge 1$.
%Suppose first that~$G$ is a trigonometric polynomial.
%Rescaling reduces
%matters
%the lemma 
%to the case
%where~$\|G\|_{p'} = 1$ and~$\|F\|_p = 1$.
%where~$\|G\|_{p'} = 1$.
%The
%desired
%norm
%estimate~\eqref{eq:Powerofnormf}
%estimate
%conclusion
%then 
%follows if~$\|f\|_p \ge 1$.
%Note that then~$\|F\|_{p} = 1$,
%and we need to prove 
%that~$\|f\|_{p} \ge 1$. 
%Since~$\hat G \ge 0$,
%Now
%\begin{gather}
%\begin{split}
%\sum_n {\hat G(n)}\hat F(n)
%= \sum_n \overline{\hat G(n)}\hat F(n) \\
%=  \frac{1}{2\pi}\int_{-\pi}^{\pi} \overline{G(t)}F(t)\,dt
%= \frac{1}{2\pi}\int_{-\pi}^{\pi} |G(t)|^{2j}\,dt
%= 1.
%\end{split}
%\end{gather}
%Of course, we can restrict the summations above to the integers~$n$ in the support,~$S_G$ say, of~$\hat G$.
%
%Let~$c = \widehat{F'}$.
Let~$c = \widehat{F'}$;
%Let~$c = \hat F$, and~$g = G$;
%then~$|c| \le |\hat f|$ on the support of~$\hat g$. 
%Since~$|\hat f| \ge |c|$ on~$S_g$, we can factor~$c(n)$ on that set as~$\varepsilon(n)\overline{\hat f(n)}$, where~$|\varepsilon(n)| \le 1$ for all~$n$ in~$S_g$.
%Let~$k$ be the trigonometric polynomial for which~$\hat k(n) = \varepsilon(n)\hat g(n)$ when~$n \in S_g$, and~$\hat k(n) = 0$ otherwise.
%As in the calculation~\eqref{eq:Switch},
%\[
%1= 
%\sum_{n \in S_g} {\hat g(n)}c(n)
%= \sum_{n \in S_g} {\hat k(n)}\overline{\hat f(n)}
%\]
%\[
%= \frac{1}{2\pi}\int_{-\pi}^\pi
%k(\theta)\overline{f(\theta)}\,d\theta.
%\]
%Since~$\|k\|_{p'} \le \|g\|_{p'} \le 1$,
%applying H\"older's inequality
%yields the desired conclusion
%that~$\|f\|_{p} \ge 1$.
%Moreover,~$g \in R(c)$
%It does belong, 
%Moreover~$g$ is a trigonometric polynomial and~$\hat g$ vanishes off the support of~$c$.
then
%Then
%Now
%As in the display~\eqref{eq:Switch},
%support of~$c$, and
Remark~\ref{rm:VanishOnSupportOfG} applies, and yields that
\begin{equation}
\label{eq:BoundSum}
%\[
\sum_n\widehat{F'}(n)\hat g(n)
= \sum_n |c(n)|\hat g(n)
\le \|f\|_p\|g\|_{2j}.
%\]
\end{equation}
%\[
The sum on the left above is equal to
\begin{equation}
%\begin{gather}
%\begin{split}
\label{eq:ToNormPower}
\begin{gathered}
%\notag
%\sum_n |c(n)|\hat g(n)
%= \sum_n \widehat{F'}(n)\overline{\hat g(n)}
%\\
%\]\[
%\notag
\frac{1}{2\pi}\int_{-\pi}^{\pi} F'(\theta)\overline{g(\theta)}\,d\theta
= \frac{1}{2\pi}\int_{-\pi}^{\pi}
\left\{
\left[\overline{g(\theta)}\right]^{j-1}
g(\theta)^j
\right\}
\overline{g(\theta)}
\,d\theta
\\
%\notag
= \frac{1}{2\pi}\int_{-\pi}^{\pi}|g(\theta)|^{2j}\,d\theta
= (\|g\|_{2j})^{2j}.
%\end{split}
%\end{gather}
\end{gathered}
\end{equation}
So~$(\|g\|_{2j})^{2j}
\le \|f\|_p\|g\|_{2j}$,
and inequality~\eqref{eq:Powerofnormf}
holds.
%follows.
%\]
%So~$g \in R(c)$.
%As noted in Remark~\ref{rm:VanishOnSupportOfG},
%it follows that~$\|g\|_{p'} \ge 1/\|f\|_p$.
%Since~$\|g\|_{p'} =1 $ by assumption,~$\|f\|_p \ge 1$ as required.
%~$1/\|f\|_p \le 1$,
%and~$\|f\|_p \ge 1$ as required.

For
Corollary~\ref{th:minimalconjugate},
%the first corollary,
let~$H$ be
%suppose that~$H$ is
%apply Theorem~\ref{th:Dual Norm Inequality} when~$G$ is
a~$p$-conjugate
%the~$p$-conjugate~$H$
of~$f$,
%let~$G = H$
and let~$J = (\overline H)^{j-1}H^j$.
%and~$F = (\overline H)^{j-1}H^j$,
Then~$\hat J$ and~$|\hat f|$ agree on the support of~$\hat H$.
It
follows that
if~$K$ is any majorant of~$f$, then~$\hat K \ge \hat J$ on the support of~$\hat H$.  Applying
Theorem~\ref{th:Dual Norm Inequality}
%the theorem
with~$F$
%temporarily
replaced by~$J$ 
and~$f$
%and~$f$
replaced
temporarily
by~$K$ yields that~$\|J\|_p \le \|K\|_p$.
%but replace the function~$f$ in the lemma with any majorant,~$J$ say, of the function~$f$ in the corollary.  Then~$\hat F \le \hat J$ on the support of~$G$, because~$\hat F = |\hat f|$ on that support, and~$J$ majorizes~$f$.  
%Therefore,~$\|F\|_{p} \le \|J\|_{p}$ for all majorants~$J$ of~$f$.  
Because~$J$
%Since~$J$
is a majorant
of the
original
function~$f$ in the corollary,
%orginal
%function~$f$,
%of~$f$,
it must be the minimal one in~$L^p$.
%Since~$F$ is a majorant of~$f$, it must be the minimal one in~$L^p$.
Since~$J$ is
then
unique, so is the function~$H$ in~$L^{2j}$ obtained
by letting~$H = |J|^{1/(2j-1)}\sgn(J)$.
%Since that majorant is unique,
%so is~$H = |J|^{1/(2j-1)}\sgn(J)$.

Corollary~\ref{th:ownmajorant} follows, because~$H$ is a~$p$-conjugate of~$(\overline H)^{j-1}H^j$.
%\end{proof}

Combining
%the
Corollary~\ref{th:minimalconjugate} with Theorem~\ref{th:trig-th}, yields that a function~$F$ in~$L^p$
%with nonnegative coefficients
is a minimal majorant of
%some
a
function~$f$ in that space if and only
if~$|F|^{1/(2j-1)}\sgn(F)$
%if the 
%function~$|F|^{1/(2j-1)}\sgn(F)$
%function~$G = |F|^{1/(j-1)}\sgn(F)$
is a~$p$-conjugate of~$f$.  Similarly, the functions~$F$ in~$L^p$ that are minimal majorants of something in~$L^p$ are those for which the coefficients
of~$|F|^{1/(2j-1)}\sgn(F)$
%of the related function~$G$
%of~$G$
are nonnegative.

There are several ways
%This makes it easy 
to see
%now follows easily
%In the rest of this section, we comment on the fact
%It follows
that some functions in~$L^p$ have nonnegative coefficients,
but 
%do not have the form specified in Theorem~\ref{th:trig-th},and so 
are not minimal majorants in~$L^p$ of anything.
%we 
%discuss
%mention
%comment on 
%other ways to reach that conclusion.
As in Remark~\ref{rm:NoAccident} with the 
r\^oles
%roles
of~$p$ and~$p'$ exchanged, one can appeal to the fact~\cite{Ru} that if~$p$ is not an even integer,
then there exist functions,~$F$ say, in~$L^p$ with nonnegative coefficients for which some 
coefficient
%coefficients
of~$|F|^{1/(2j-1)}\sgn(F)$
%of~$|F|^{p/p'}\sgn(F)$
%of~$G$, that is~$|F|^{p/p'}\sgn(F)$,
%of~$|F|^{p/p'}\sgn(F)$
fails
%fail
to be nonnegative.

It is simpler, however,
%In the cases of interest here, however, it
%is
%seems
%simpler
to consider the supports,~$S_j$ say, of the transforms of 
products~$(\overline G G)^{j-1}G$
when~$G$
is a trigonometric
polynomial for which~$\hat G \ge 0$.
%polynomial.
%with
%has
%at least two nonzero coefficients.
%products~$(\overline G)^{j-1}G^j$, where~$j > 1$.
If~$G$
%is a trigonometric polynomial
has
%with
at most one nonzero coefficient, then 
the same is true for~$(\overline G)^{j-1}G^j$ when~$j > 1$.
%those products
%the supports of the coefficients of the products~$(\overline G)^{j-1}G^j$, where~$j > 1$ 
%also have at most one nonzero coefficient.
%In the remaining cases, write the product as~$(G\overline G)^{j-1}G$.
%For
%trigonmetric
%other 
%polynomials~$G$ with nonnegative coefficients, %however,
%Denote
%denote
%that
%the
%support
%of~$(\overline G)^{j-1}G^j$ by~$S_j$.
%
In the remaining cases,
%however,
since
%Since
%The
%Then
the
support of the coefficients of~$\overline GG$
contains all differences of members of~$S_1$,
%
%Hence,
that support
%it
contains~$0$
as well as some positive integer and its negative.
Then~$S_2$
%So~$S_2$
%So the support of the coefficients of~$(G\overline G)G$ %
%each of these supports includes the previous one, and 
includes~$S_1$,
and it
also
contains an integer to left of~$S_1$ and an integer to the right of~$S_1$.
%the latter
%that support
%and an integer to the right of it.
Similarly,~$S_{j+1}$ contains at least two more points than~$S_j$.
In particular,~$S_j$
%In particular, the support of~$(\overline G)^{j-1}G^j$ 
must
%then
contain at least~$4$ integers when~$j > 1$.
It follows that any
%So a
trigonometric polynomial~$F$ with  exactly~$2$ or~$3$ nonzero coefficients 
cannot be factored in the form specified in Theorem~\ref{th:trig-th}.

%We continue to suppose that~$p'$ is an even integer.
%Fix an integer~$j >1$,  let~$p=2j$, and let~$f$,~$F$ and~$G$ be the functions mentioned in Theorem~\ref{th:trig-th}.
%and denote the support of~$\hat f$ by~$S_f$.
This approach 
also 
provides
many
%an
%We 
%return to the unit circle, and
%first 
%discuss the 
%abundance of
examples where the support of
the 
coefficients 
%of~$G$
of the function~$G$ in Theorem~\ref{th:trig-th} 
is strictly smaller
than the support of~$\hat f$.
Just
let~$f$ be any exact minorant of~$(\overline G)^{j-1}G^j$,
%let~$f = (\overline G)^{j-1}G^j$,
where~$G$ is a trigonometric polynomial with nonnegative coefficients, with at least two of them 
different
%differing 
from~$0$.
%Similarly, one can use any function~$G$ in~$L^{2j}$ with nonnegative coefficients,
%and with~$S_G$ equal to a set like~$\{n(n+1)/2\}_{n=1}^\infty$ that omits many integers, but for which~$S_G-S_G$ contains all integers.
In particular,~$G$
%Then~$G$
%In all these cases,~$G$
cannot be a multiple of the exact majorant~$E_f$ of~$f$,
because~$\hat G$ vanishes on part of the support of~$\widehat{E_f}$.
%Then we  comment on the fact 
%Before that, we explain how the new part of Theorem~\ref{th:trig-th},
%about the support of~$\hat G$, is implicit
%in the previous duality proofs~\cite{HaL, DGLPQ}.
%And we use our analysis of the classical proof
%to show that factorizations using other~$p$-conjugates,
%if any, yield
%majorants with small enough~$L^p$ norms when~$p'$ is even.
%And we relate our duality proof to the one in~\cite{DGLPQ}.
%
%Finally, we discuss the fact that summability
%plays only a minor role in our proof of the duality theorem.

We now
%Finally, we
discuss the cases where~$\|F\|_p = \|f\|_p$
in 
Theorem~\ref{th:Dual Norm Inequality}.
%Theorems~\ref{th:trig-th} and~\ref{th:Dual Norm Inequality}.
They are the ones where the~$p$-conjugate~$G$ has the property that
the
%inequality~$(\|G\|_{2j})^{2j-1} \le \|f\|_p$
inequality
\begin{equation}
\label{eq:Strict?}
%\[
(\|G\|_{2j})^{2j-1} \le \|f\|_p
%\]
\end{equation}
is not strict.
When~$f$ is a trigonometric polynomial, so is~$G$, and the relations~\eqref{eq:BoundSum}
and~\eqref{eq:ToNormPower} apply when~$g = G$ and~$F' = F$.
%Then
%In that case,
Matters then
%matters
%then
reduce to determining when the inequality
\[
\sum_n |c(n)|\hat g(n)
\le \|f\|_p\|g\|_{2j}
\]
is not strict.

In the proof of that inequality,
%In its proof,
%In our proof of the latter,
%we considered certain trigonometric polynomials~$g$
%with nonnegative coefficients,
%rescaled them to have norm~$1$ in~$L^{p'}$,
%where~$\|g\|_{p'} = 1$,
%and
we
used Remark~\ref{rm:VanishOnSupportOfG}.
It concerns the equations~\eqref{eq:Switch}
and the inequalities~\eqref{eq:minorant},
in which
the
%a
sequence~$c$ is factored
on the support of~$\hat g$
%in which~$c$ is factored
as~$\varepsilon\overline{\hat f}$,
%as~$\overline\varepsilon\hat f$,
where~$|\varepsilon| \le 1$ on
%where~$|\varepsilon(n)| \le 1$ for all~$n$ in
that support,
%the support of~$\hat g$,
and~$k$ is
the minorant of~$g$
%the trigonometric polynomial
for which~$\hat k = \varepsilon\hat g$
%on that support.
there.
%In that context,~$\|g\|_{2j}$ was set equal to~$1$, but that restriction was not used until later, in getting the lower bound for~$K_p(c)$.
%
The inequalities
in line~\eqref{eq:minorant}
%in~\eqref{eq:minorant}
%Those inequalities
are the instance
of H\"older's inequality where
\begin{equation}
\label{eq:RepriseHolder}
\left|\frac{1}{2\pi}\int_{-\pi}^\pi
k(\theta)\overline{f(\theta)}\,d\theta\right|
\le \|k\|_{2j}\|f\|_{2j/(2j-1)},
%\le \|g\|_{p'}\|f\|_p,
\end{equation}
and the
%elementary
fact that
\begin{equation}
\label{eq:ReMajorant}
%\[
\|k\|_{2j} \le \|g\|_{2j}
%\quad\text{when~$|\varepsilon(n)| \le 1$ for all~$n$ in the support of~$\hat g$.}
\end{equation}
%\]
%when~$|\varepsilon(n)| \le 1$ for all~$n$ in the support of~$\hat g$.
when~$|\varepsilon(n)| \le 1$ for all~$n$ in the support of~$\hat g$.  The norms of~$f$
and~$(\overline gg)^{(j-1)}g$
%and~$|g|^{2(j-1)}g$
in~$L^{2j/(2j-1)}$ are equal if and only if equality holds in
both
%the two
inequalities displayed
just
above.
%both statements above.

%Now
%It is easy to see that 
Inequality\eqref{eq:ReMajorant}
%inequality\eqref{eq:ReMajorant}
%the upper majorant inequality above
%second
%one
%inequality
is strict
%if~$ |\varepsilon(n)| < 1$ for some~$n$ in the support of~$\hat g$.
%
if the inequality~$|\varepsilon(n)| \le 1$
is strict for some~$n$ in the support of~$\hat g$.
It follows that if~$g$ and~$k$
%So~$g$ and~$k$
%So the only cases where~$g$ and~$k$ might
have the same norms
in~$L^{2j}$,
%in~$L^{2j}$
then~$|\varepsilon| = 1$
%are those where~$|\varepsilon| = 1$
on the support of~$\hat g$.
That is~$|\hat f| = |c|$ on that set, and~$k$ is an exact minorant of~$g$.

In nontrivial cases,
inequality~\eqref{eq:RepriseHolder}
%Inequality~\eqref{eq:RepriseHolder}
is strict unless~$f$
%factors as~$(k\overline k)^{j-1}k$
factors as~$(\overline k)^{j-1}k^j$
%is equal to~$(k\overline k)^{j-1}k$
multiplied by some constant,~$C$ say.
Then~$C > 0$
%in nontrivial cases,
%Then~$C > 0$,
because
\[
\overline C\|k\|_{2j}^{2j}
=
\frac{1}{2\pi}\int_{-\pi}^\pi
k(\theta)\overline{f(\theta)}\,d\theta
= \sum \hat g(n)|c(n)| > 0.
\]
Now~$\|f\|_p = C(\|k\|_{2j})^{2j-1} = C(\|g\|_{2j})^{2j-1}$,
while~$\|F\|_p = (\|g\|_{2j})^{2j-1}$.
So~$C=1$
%It follows that~$C=1$
%It follows that~$C=1$,
%Since~$f$ and~$F$ have the same norms
if~$f$ and~$F$
%since~$f$ and~$F$
have the same norms 
in~$L^p$.
%in~$L^{2j/(2j-1)}$,
%in~$L^{2j/(2j-1)}$.
%it follows that~$C=1$.
Then~$f$ does indeed factor
as~$(\overline k)^{j-1}k^j$
where~$k$ has the same norm in~$L^{2j}$ as its exact majorant.

For the converse,
%Conversely,
suppose that
%if
a trigonometric polynomial~$k$ has
the same~$L^{2j}$ norm as its exact
majorant~$g$.
Lettng~$f = (\overline k)^{j-1}k^j$,
%Then let~$f = (\overline k)^{j-1}k^j$,
%Let~$f = (\overline k)^{j-1}k^j$,
and~$F = (\overline g)^{j-1}g^j$
%and~$F = (\overline g)^{j-1}g^j$.
then
gives~$F$
%Then~$F$
the same~$L^p$ norm as~$f$.
%Now~$F$ clearly majorizes~$f$.
Now~$F$ clearly majorizes~$f$,
%the same~$L^p$ norm as~$f$,
and
it
%It
%To confirm that it
is
%and~$F$ is
minimal
in~$L^p$
%majorant of~$f$ in~$L^p$
if~$g$ is a~$p$-conjugate of~$f$.
It is easy to see that~$g$ satisfies the first two conditions in Definition~\ref{def:conjugate}.
Argue
%Now
%For the third condition,
%argue
%rewrite the~$2j$-th powers of the~$L^{2j}$ norms of~$g$ and~$f$
as in lines~\eqref{eq:BoundSum}
and~\eqref{eq:ToNormPower} to get that
%and~\eqref{eq:ToNormPower}.
%That is
\[
(\|g\|_{2j})^{2j} =
\sum_n \hat F(n)\hat g(n),
\quad\textrm{and}\quad
(\|f\|_{2j})^{2j} =
\sum_n \hat f(n)\overline{\hat k(n)}.
\]
Since
these
%the
norms
%above
are equal, so are the sums.
Since
\[
|\hat f(n)\overline{\hat k(n)}|
\le \hat F(n)\hat g(n)
\]
for all~$n$,
the two sides above are equal.
%these inequalities are not strict.
Since~$|\hat k(n)| = \hat g(n)$,
the quantities~$|\hat f(n)|$ and~$\hat F(n)$ agree on the support
%of~$\hat g$,
of~$\hat g$.
%and~$g$
So~$g$
satisfies the
third condition
%first three conditions
in Definition~\ref{def:conjugate}.
It was shown
%As
in
Section~\ref{sec:OtherProofs} that
%Section~\ref{sec:OtherProofs},
the first three conditions
in the definition
%they
imply
the remaining
one.
%condition.

%majorant~$g$, and if~$f$ factors as above, then~$g$ is a~$p$-conjugate for~$f$,
%and the minimal majorant of~$f$ in~$L^p$ has
%the same~$L^p$ norm as~$f$.
%
%Therefore,
%It follows that,
%in the notation of the proof of Theorem~\ref{th:Dual Norm Inequality},
%if~$F = |g|^{j-1}g$ has the same norm in~$L^{2j/(2j-1)}$ as~$f$,
%and if~$|\hat f| \ge \hat F$
%on the support of~$\hat g$,
%then~$|\hat f| = \hat F$ on that support,
%and~$f$ factors as above.
%Moreover,~$|g|^{j-1}g$
%Moreover,~$F = |g|^{j-1}g$
%then majorizes~$|k|^{j-1}k$, so that it majorizes~$f$ too.

As in~\cite{HaL} and~\cite{Ba}, both
%Both
directions in the
%The
reasoning above
extends
to general functions~$k$ and~$g$ in~$L^{2j}$, because
%the formula
formulas like
\[%(1/2\pi)
\frac{1}{2\pi}
\int_{-\pi}^\pi
k(\theta)\overline{f(\theta)}\,d\theta
= \sum_n \hat k(n)\overline{\hat f(n)}
\]
do.
%extend to all functions~$k$ in~$L^{p'}$
%and~$f$ in~$L^p$ when~$1 < p' < \infty$.
%For general exponents~$p'$ in the interval~$(1, \infty)$, this is well known by norm summability or the M. Riesz inequality, and the cases where~$p'$ is even are true for simpler reasons.
%simpler.
%The cases where~$p'$ is even are simpler.

Finally, we
%discuss some cases where~$\|F\|_p = \|f\|_p$
%in 
%Theorem~\ref{th:Dual Norm Inequality}.
%To that end,
%Finally,
%we
ask when
%consider the cases where
%possibility for
a function,~$G$ say, in~$L^{2j}$, with nonnegative coefficients
has
%that it have
the same norm in~$L^{2j}$ as all functions that it majorizes exactly.
This is equivalent to requiring
for all exact minorants~$k$ of~$G$,
that~$k^j$
have
%has
the same norm in~$L^2$ as~$G^j$.
Since~$k^j$ is majorized by~$G^j$,
%Since~$\left|\widehat{g^j}\right| \le \widehat{G^j}$,
that happens if and only if the coefficients of~$k^j$ all have the same absolute value as those of~$G^j$.

The support,~$T_j$ say,
%The support,~$S_j$ say,
of~$\widehat{G^j}$ is the algebraic sum of~$j$ copies of the support,~$S$ say, of~$\hat G$.  
An
%So an
integer,~$n$ say,
belongs to~$T_j$
%belongs to~$S_j$
if and only if there is a function~$\alpha$ mapping~$S$ into the nonnegative integers 
for which
%so that
\begin{equation}
\label{eq:Order j}
%\[
\sum_{k \in S}\alpha(k) = j,
\quad\text{and}\quad
n = \sum_{k \in S} \alpha(k)k.
%\]
\end{equation}
As in \cite{KovTan}, call~$S$ a~$B_j$ set, or a Sidon set of order~$j$, if for each integer~$n$ 
in the set~$T_j$
%in the set~$S_j$
there is at most one function~$\alpha$ with the properties specified above.
%The
It is easily
to see
%seen
that
%For instance,
the range of a
sequence~$(\lambda_i)$
%sequence~$(n_i)$
of positive integers
for which~$\lambda_{i+1} > (1+\varepsilon)\lambda_i$
%which~$n_{i+1} > (1+\varepsilon)n_i$
for all~$i$ has this property if~$\varepsilon$ is large
enough;
%enough, and more
%enough.
more
%More
subtle
%known
examples were used in~\cite[\S 4]{RuT}.
As explained in the survey~\cite{O'B},
some
%many
authors use the
term~$B_j$
%term~$B_h$
set with slightly different meanings, but the one above is %also
%widely used.
widely used, and
it
fits our situation.
%There is another notion of ``Sidon set''

A
%Indeed a
function~$G$ in~$L^{2j}$ with nonnegative coefficients has the same norm as all its exact minorants if and only if the support of~$\widehat{G}$ is a Sidon set of order~$j$.
This is likely well known for exponential
sums.
%sums, and
%by
%the same
%reasoning
To confirm
that
it
%this
extends to other functions,
%Suppose
%To verify this,
suppose
first that~$\hat G$ vanishes off such set.  Then, in the context of
the
formulas~\eqref{eq:Order j},
the coefficient~$\widehat{G^j}(n)$
is a sum of
a suitable number of
copies of
the
same
product
\begin{equation}
\label{PowerProduct}
%\[
%\widehat{G^j}(n) =
\prod_{k \in S}\hat G(k)^{\alpha(k)},
%\]
\end{equation}
and similarly for~$\widehat{k^j}$
for each exact minorant~$k$ of~$G$.
Hence~$\widehat{k^j}$ and~$\widehat{G^j}$ have the same absolute values, and~$\left\|{k^j}\right\|_2 = ~\left\|{G^j}\right\|_2$.

On the other hand,
if the support of~$\hat G$ is not a Sidon set of order~$j$, then there is an integer~$n$ in the support of~$\widehat{G^j}$ for which
the
second
%first
of the equations~\eqref{eq:Order j}
holds for more than one choice of~$\alpha$.
Then~$\widehat{G^j}(n)$
%which~$\widehat{G^j}(n)$
is
again
equal to a sum of 
copies of
products
%more than one
%type of
%product
of the form
displayed
%given
%appearing 
in
%equation~\eqref{PowerProduct}. 
line~\eqref{PowerProduct},
%line~\eqref{PowerProduct}. 
%in~\eqref{PowerProduct}.
but with different~$\alpha$'s being used in some
products.
%terms.
%cases.
As before,~$\widehat{k^j}(n)$
%In that case,~$\widehat{g^j}(n)$
is equal to a sum of similar products, 
but now
%and
the arguments of the coefficient of~$k$ can be chosen
to create
%so that there is
some cancelation in that sum.
%That
%This
Then~$\left|\widehat{k^j}(n)\right|
%makes~$\left|\widehat{k^j}(n)\right|
< \widehat{G^j}(n)$,
%Then~$\left|\widehat{g^j}(n)\right|
%< \widehat{G^j}(n)$,
and~$\|k\|_{2j} < \|G\|_{2j}$.

\begin{remark}
\label{rm:LopezRoss}

The term ``Sidon set" is also widely used
with a very different
%another
%different
meaning, as in~\cite{LoR} and~\cite{GrHa}.
For clarity, those
%Those
sets could
%also
be called ``Sidon interpolation sets."
Among their many
%interesting
properties is the fact that for functions whose transforms vanish 
off such a set
%off the set,
each~$L^p$
%norm
norm,
where~$p \in [1, \infty)$,
%where~$p \in [1, \infty)$
is equivalent to the~$L^2$ norm.  So the exact majorant of every such function in~$L^p$ also belongs
to~$L^p$,
%to~$L^p$;
and~$\|E_f\|_p \le C\|f\|_p$
%moreover,~$\|E_f\|_p \le C\|f\|_p$
in this case, where~$C$ depends on~$p$ and the choice
of Sidon interpolation set, but not on~$f$.
%of Sidon set.
Both notions of Sidon set are discussed in~\cite[\S 4]{RuT},
but in that context the term just meant Sidon interpolation set.

\end{remark}

\bibliographystyle{amsplain}

\end{document}